\documentclass[10pt]{amsart}
\usepackage{amssymb,amsmath,extarrows}
\usepackage{stmaryrd}
\usepackage{algorithm}
\usepackage{algorithmic}
\usepackage{mathbbol, bm, tensor}
\usepackage{color}
\usepackage{epsfig}
\usepackage{ulem}
\usepackage{subfigure}
\usepackage[graphicx]{realboxes}
\usepackage{overpic}
\usepackage[margin=30mm]{geometry}
\usepackage[numbers,sort&compress]{natbib}

\newcommand{\ep}{\varepsilon}
\newcommand{\mr}{\mathring}

\newcommand{\vertiii}[1]{{\left\vert\kern-0.25ex\left\vert\kern-0.25ex\left\vert #1
    \right\vert\kern-0.25ex\right\vert\kern-0.25ex\right\vert}}
	
\makeatletter \@addtoreset{equation}{section} \makeatother

\theoremstyle{plain}
\newtheorem{theorem}{Theorem}[section]
\newtheorem{lemma}{Lemma}[section]
\newtheorem{proposition}{Proposition}[section]

\theoremstyle{definition}
\newtheorem{remark}{Remark}[section]

\title[The coupled fluid flow with friction-type interface condition]{The well-posedness and regularity of the Non-stationary Stokes and Navier-Stokes equations with the friction-type interface condition}

\author[Qi Wang]{Qi Wang}
\address{Institute of Fundamental and Frontier Sciences, University of Electronic Science and Technology of China, Chengdu, 610051, China}
\email{qi\_wang@std.uestc.edu.cn}
\author[Takahito Kashiwabara]{Takahito Kashiwabara}
\address{Graduate School of Mathematical Sciences, The University of Tokyo, Tokyo, Japan}
\email{tkashiwa@ms.u-tokyo.ac.jp}
\author[Guanyu Zhou]{Guanyu Zhou}
\address{Institute of Fundamental and Frontier Sciences, School of Mathematical Sciences, University of Electronic Science and Technology of China, Chengdu, 610051, China}
\email{wind\_geno@live.com, zhoug@uestc.edu.cn}

\date{\today}
\subjclass[2010]{}
\thanks{The research of Takahito Kashiwabara was supported by JSPS Grant-in-Aid for Early-Career Scientists (No. 20K14357). The research of Guanyu Zhou was partially supported by NSFC General Project (No. 12171071) and the Natural Science Foundation of Sichuan Province (No. 2023NSFSC0055).}

\begin{document}
% abstract
\begin{abstract} 
The friction-type interface condition (FIC) is introduced to describe the phenomenon of the slip and leak of fluid flow on the interface happens only when the difference of stress force is above a threshold. 
The FIC involves the subdifferential and can be regarded as an intermediate form of the Dirichlet and the Neumann boundary conditions. 
This work is devoted to the well-posedness of the non-stationary (Navier-)Stokes equations with FIC in 2D and 3D, the weak forms of which are parabolic variational inequalities of the second type. 
We establish the existence theorems using the regularization technique and the Galerkin method. For the Stokes case, we prove the global unique existence and investigate the $H^2$ regularity. 
In the case of 2D Navier-Stokes equation, we show the global unique existence of the weak and strong solutions, respectively. 
For the Navier-Stokes case in 3D, we demonstrate the global existence of the weak solution and the local unique existence of the strong solution. 
\end{abstract}
\maketitle
%
% \tableofcontents
%%
\section{Introduction}\label{sec:1}
%% friction-type
The friction-type leak/slip boundary conditions (FLBC/FSBC) were first introduced by Fujita in ~\cite{Fujita94} for the incompressible viscous flow, which can be utilized to simulate the wave-shore interaction ~\cite{Kawarada98} and the spilled oil ~\cite{Suito04}.
For the (non-)stationary (Navier-)Stokes equations with FLBC/FSBC, the well-posedness, regularity and numerical simulation theories are developed by ~\cite{Fujita02, Fujita02-1, Fujita95, Kashiwabara13-1, Kashiwabara13-2, Kashiwabara13, Saito04}.
A deficiency of FLBC in modeling is noted in ~\cite{zhou23}. 
They point out that if the leakage occurs, by the divergence-free condition, both outflow and inflow must be observed on the boundary. 
For the nonstationary Navier-Stokes equations with FLBC, the energy inequality and the global-in-time weak solution in 3D or strong solution in 2D are unattainable due to the unprescribed inflow.
Therefore, it is more appropriate to model the problem by the friction-type leak interface condition (FLIC):
\begin{equation}\label{eq:leak}
  [\bm{u}_n]=0,~\bm u_\tau|_{\Omega_{\mathrm{in}}}=\bm u_\tau|_{\Omega_{\mathrm{out}}}=0, ~|[\bm{\sigma}(\bm{u}, p) \bm n]\cdot \bm n| \leq g, ~[\bm{\sigma}(\bm{u}, p) \bm n\cdot \bm n]\bm{u}_n-g|\bm{u}_n| =0. 
\end{equation}
Actually, ~\cite{Conca87, Conca88} have proposed the stationary Stokes-Stokes coupled flow with the FLIC to model the viscous fluid through a perforated membrane, where the leakage occurs only when the stress difference on the membrane is above a threshold.
The well-posedness of this model and the limited behavior of the periodically distributed leak parts are investigated in ~\cite{Maris12}. 
However,  the well-posedness has not been proved for the non-stationary Stokes/Navier-Stokes with friction-type interface condition.

%% present results
The existence and uniqueness of weak solution for stationary Stokes equations with FSBC and FLBC are established in ~\cite{Fujita94}. The generalized FSBC is considered in ~\cite{Roux05, Roux07}.
Moreover, the $H^2$-$H^1$ regularity is proved in ~\cite{Saito04}, which can be used to analyze the time-dependent problems. 
For non-stationary cases, ~\cite{Fujita01, Fujita02} study the strong solution of Stokes models with FLBC/FSBC by using nonlinear semigroup theory.
The weak and strong solutions of the Navier-Stokes problems under FLBC/FSBC are proposed in ~\cite{An09, Kashiwabara13}
where they used the nonlinear variational inequality theory ~\cite{Lions67, Glowinski84}, regularization method and Galerkin approximation ~\cite{Temam77}.
We refer the reader to ~\cite{Ayadi10, Fujita95, Kashiwabara13-1, Kashiwabara13-2,Li2008, Li10, Li11} for the numerical methods of the (Navier-)Stokes equations under the FLBC/FSBC.
The motivation of our work is to study the well-posedness of the non-stationary (Navier-)Stokes equations with the FIC.

%% PDE model
Let $\Omega_{\mathrm{in}}$ and $\Omega_{\mathrm{out}}$ be two smooth bounded domains in $\mathbb{R}^d$ $(d=2,3)$ with a common closed interface $\Gamma = \overline{\Omega_{\mathrm{in}}} \cap \overline{\Omega_{\mathrm{out}}}$ (see Fig.~\ref{fig:domain}). 
Set $\Omega_a := \Omega_{\mathrm{in}} \cup \Omega_{\mathrm{out}}$, $\Omega = \Omega_a \cup \Gamma$, $\Gamma_D := \partial\Omega_{\mathrm{out}}\setminus \Gamma$. 
Let $(\bm u, p)$ be the velocity and pressure of the fluid flow. 
The stress and deformation tensors are defined as 
\[ 
  \bm{\sigma}(\bm{u}, p):=2 \nu \mathbb{D}(\bm{u})-p\mathbb{I}, \quad \mathbb{D}(\bm{u}):=(\nabla \bm{u}+\nabla^{\top} \bm{u})/2. 
\]
The Navier-Stokes problem {\bf(NS-P)}  is stated as follows: 
\begin{subequations}\label{eq:NS-P}
  \begin{align}
    \bm{u}' +(\bm{u} \cdot \nabla) \bm{u}-\nabla \cdot \bm{\sigma}(\bm{u}, p)&= \bm f \quad &&\text { in } \Omega_a \times(0, T), \label{eq:P-a}\\
    \nabla \cdot \bm{u}&=0 \quad &&\text { in } \Omega\times(0, T), \label{eq:P-b}\\
    \bm{u}&=\bm{u}_{0} \quad &&\text { in }  \Omega,\label{eq:P-c} \\
    \bm{u}&=0 \quad &&\text { on } \Gamma_{D} \times[0, T], \label{eq:P-d} \\
    [\bm{u}]=0, \quad|[\bm{\sigma}\bm n]| \leq g, \quad[\bm{\sigma} \bm n] \cdot \bm{u}-g|\bm{u}|&=0 \quad &&\text { on } \Gamma \times[0, T].\label{eq:P-e}
  \end{align}
\end{subequations}
Here $\bm n$ is the unit outer normal vector on $\Gamma$ towards $\Omega_{\mathrm{out}}$ and $\bm{\sigma} \bm n$ is the stress vector. 
The threshold function $g$ is a positive function on $\Gamma$. $[A] := A|_{\Omega_{\mathrm{out}}}- A|_{\Omega_{\mathrm{in}}}$ represents the jump across $\Gamma$.

%%%
The Stokes problem {\bf(S-P)} is obtained by omitting $(\bm{u} \cdot \nabla) \bm{u}$, i.e., replacing \eqref{eq:P-a} with 
\begin{equation}\label{eq:S-P}
  \bm{u}' -\nabla \cdot \bm{\sigma}(\bm{u}, p)= \bm f \quad \text { in } {\Omega_a} \times(0, T).
\end{equation}

In view of $[\bm u]= 0$, $\bm u$ is continuous on $\Gamma$. Thus, we have $\bm u \in \bm H^1(\Omega)$, but we do not expect $\bm u \in \bm H^2(\Omega)$. The FIC ~\eqref{eq:P-e} describes the coupled fluid flows with a continuous velocity on the interface,  whereas the jump of the stress vector across $\Gamma$ has an upper bound $g$.  We see that 
\[
\begin{cases}
    \bm{u}=0  &\text { on }\{(x, t) \in \Gamma \times(0, T) : [\bm{\sigma n}]| < g\},\\
    [\bm{\sigma}\bm n]=g \frac{\bm{u}}{|\bm{u}|} &\text { on }\{(x, t) \in \Gamma \times(0, T) : \bm{u} \neq 0\}.
\end{cases}
\]
When the threshold $g$ is sufficiently large, one can expect the homogeneous Dirichlet condition to take place on $\Gamma$,  i.e., $\bm u|_{\Gamma}=0$. 
An extreme case is $g=\infty$, which implies $\bm u=0$ on $\Gamma$. 
On the contrary, if $g=0$, the interface condition becomes the Neumann condition $[\bm\sigma \bm n] = 0$ on $\Gamma$. It's interesting that such FIC connects the Dirichlet and Neumann-type conditions. 
Note that \eqref{eq:P-e} can be equivalently written in the form of subdifferential
\begin{equation}\label{eq:subd}
  [\bm{\sigma} \bm n] \in g \partial|\bm u| =
  \begin{cases}
    g\frac{\bm u}{|\bm u|} &~\bm u\neq 0 \text{ on }\Gamma,\\ 
    \{\bm z \in \mathbb{R}^d \mid |\bm z| \le g \} &~\bm u = 0 \text{ on }\Gamma.
  \end{cases}
\end{equation}

\begin{figure}[htbp] 
  \centering
  \begin{overpic}[scale=0.35]{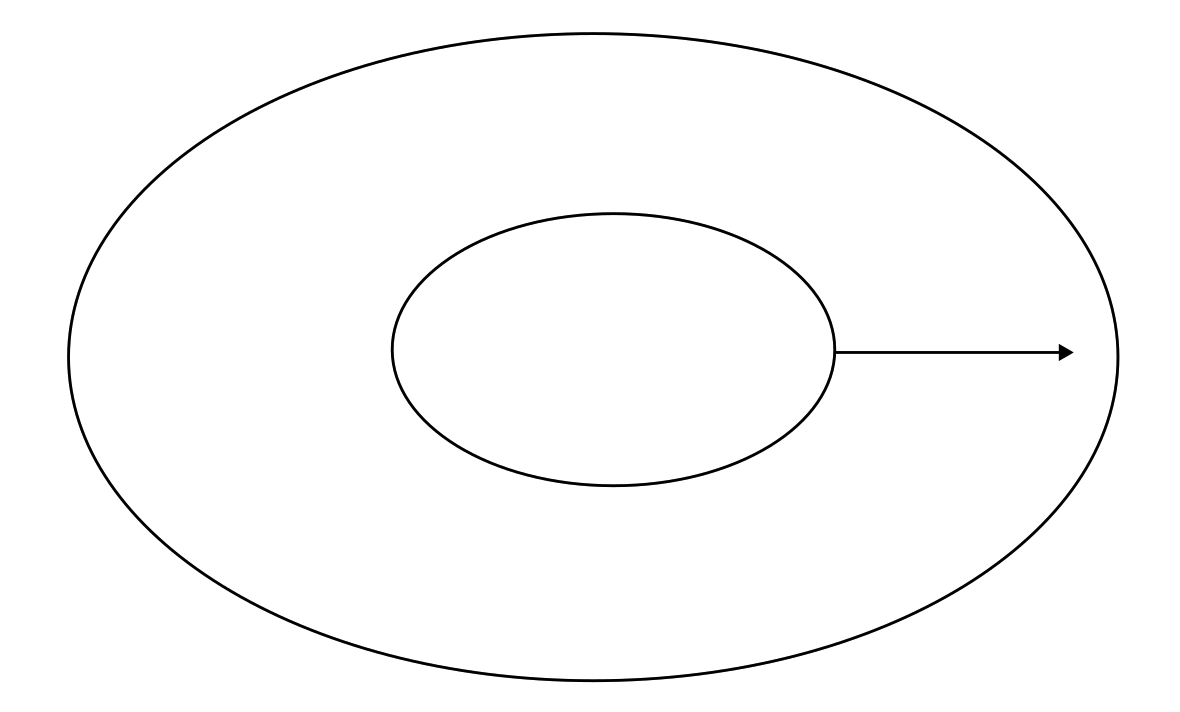}
    \put(45,29){$\Omega_{\mathrm{in}}$}
    \put(80,31){$n$}
    \put(15,29){$\Omega_{\mathrm{out}}$}
    \put(67,35){$\Gamma$}
    \put(67,5){$\Gamma_D$}
  \end{overpic}
  \caption{The regions $\Omega_{\mathrm{in}}$, $\Omega_{\mathrm{out}}$ and the interface $\Gamma$} \label{fig:domain}
\end{figure}
%

%% interface
% There are various kinds of interface conditions being proposed for the coupled fluid interaction, for instance, the Beavers-Joseph(-Saffman) condition ~\cite{Beavers67, Saffman71}, 
% the shear stress jump ~\cite{OchoaTapia1995}, the generalized jump condition ~\cite{Angot18}. They have broad applications in industrial filtration, porous medium, oil/gas reservoirs, groundwater pollution, and blood flow in the tumor.

%% our work
This paper aims to prove the well-posedness of non-stationary (Navier-)Stokes equations with FIC. First, we show the equivalence of parabolic variational inequalities and the PDE in the weak sense for the Stokes case. 
We establish the global unique existence by using the regularization technique and Galerkin's method. 
In addition, we present a detailed discussion on the uniqueness of an additive constant for pressure $p$ for the cases with or without permeation (i.e., $\bm u\cdot \bm n \neq 0$ or $\bm u\cdot \bm n = 0$ on $\Gamma$).
Moreover, we use the regularity result for $\bm u_{0\ep}$ which is the solution of stationary variational inequality at $t=0$ to prove the spatial $H^2$-$H^1$ regularity of $(\bm u,p)$.

Second, we investigate the unique existence of weak and strong solutions for the Navier-Stokes problem in 2D. 
As for the 3D case, we focus on the existence of global weak solution and the unique local existence of strong solution.
We will use the same technique to prove the results. 
The estimate of nonlinear term $(\bm u \cdot \nabla)\bm u$ is difficult because the Sobolev embedding theorem is weaker in 3D.
Our results are better than LBCF considered in ~\cite{Kashiwabara13} where only have a local-in-time solution, even a smallness assumption on $\bm u_0$ in 2D.

%% composition of the paper
The remaining sections of this paper are arranged as follows. The function spaces, notations, and some inequalities are given in Section ~\ref{sec:2}. In Section ~\ref{sec:3}, we demonstrate the global unique existence and regularity of the Stokes case. Finally, section ~\ref{sec:4} is devoted to studying the weak and strong solutions for the Navier-stokes problem.

\section{Preliminaries}\label{sec:2}
Throughout this paper, we assume that the domain $\Omega$ is smooth enough. 
The precise regularity of $\Omega$ which is required for our main theorems can be found later.
If not specified, we use the symbol $C$ to denote different positive constants dependent on $\Omega$ only.

The following part will present function spaces, bilinear and trilinear forms, and some essential inequalities. The function spaces that are vector/tensor-valued are written with bold fonts, whereas fine fonts mean scalar quantities.
\[
\begin{aligned}
  &\bm{\mathcal{D}}(\Omega):=C_0^\infty(\Omega)^d,  \quad \bm H:= \{ \text{the closure of } {\bm {\mathcal{D}}(\Omega)} \text{ in } {\bm L^2(\Omega)} : \bm{u}|_\Gamma \in \bm L^2(\Gamma)\}, \\ 
  & \bm H_\sigma:= \{ \bm v \in \bm H, ~\nabla \cdot \bm v =0 \}, \quad \bm V:= \{ \text{the closure of } {\bm {\mathcal{D}}(\Omega)} \text{ in } {\bm H^1_0(\Omega)} \},\\ 
  & \bm V_\sigma:=\{\bm v \in \bm V,  \nabla \cdot \bm v =0 \}, \quad \bm V^0 :=\{\bm v\in \bm V : \bm v=0 \text{ on }\Gamma\}, \quad \bm V_\sigma^0:= \bm V_\sigma \cap \bm V^0, \\
  & \bm H^2(\Omega_a) := \{ \bm v\in \bm V : \bm v|_{\Omega_\mathrm{in}} \in \bm H^2(\Omega_\mathrm{in}), \bm v|_{\Omega_\mathrm{out}} \in \bm H^2(\Omega_\mathrm{out})\}, \\ 
  & Q:=L^2(\Omega) ,\quad \mr{Q}:=L^2_0(\Omega)=\{q\in Q : (q,1)_{\Omega}=0\}, \\ 
  & H^1(\Omega_a) := \{ q \in \mr{Q}: q|_{\Omega_\mathrm{in}} \in H^1(\Omega_\mathrm{in}), q|_{\Omega_\mathrm{out}} \in H^1(\Omega_\mathrm{out})\}.
\end{aligned}
\]
The norms of $\bm H^2(\Omega_a)$ and $H^1(\Omega_a)$ are defined as follows:
\[
\|\bm v\|_{\bm H^2(\Omega_a)}: = \|\bm v\|_{\bm H^2(\Omega_{\mathrm{in}})} + \|\bm v\|_{\bm H^2(\Omega_{\mathrm{out}})}, \quad \|q\|_{ H^1(\Omega_a)}: = \|q\|_{ H^1(\Omega_{\mathrm{in}})}+ \|q\|_{ H^1(\Omega_{\mathrm{out}})}.
\]
The Lebesgue and Sobolev spaces on $\Gamma$ also be used. 
\[
  \bm \Lambda:= L^2(\Gamma)^d, \quad \bm \Lambda_{1/2}:= H^{\frac{1}{2}}_{00}(\Gamma)^d ,\quad \bm \Lambda_{-1/2}:= H^{-\frac{1}{2}}(\Gamma)^d, \quad \mr{\bm \Lambda}_{1/2} := \{ \bm\mu \in \bm \Lambda_{1/2} : (\bm\mu \cdot \bm n, 1)_\Gamma = 0\}.   
\]
By duality of $\bm \Lambda_{1/2}$ and $\bm \Lambda_{-1/2}$, the definition of $\|\cdot\|_{\bm \Lambda_{-1/2}}$ is 
\[
  \|\bm \lambda\|_{\bm \Lambda_{-1/2}} :=\sup_{\bm \mu \in \bm \Lambda_{1/2}}\frac{\langle \bm \lambda, \bm \mu \rangle_\Gamma}{\|\bm \mu\|_{\bm \Lambda_{1/2}}} \quad \forall \bm \lambda \in \bm \Lambda_{-1/2}.
\]
For a positive function $g$ on $\Gamma$, we define the weighted Lebesgue spaces $\bm L_g^1(\Gamma)$ and $\bm L_{1/g}^{\infty}(\Gamma)$:
\[
  \begin{aligned}
  &\bm L_g^1(\Gamma):=\{\bm \mu \in L^1(\Gamma)^d : \|\bm \mu\|_{L_g^1(\Gamma)}<\infty \} \text{ where } \|\bm \eta\|_{\bm L_g^1(\Gamma)}=\int_{\Gamma} g|\bm\eta | ~d s, \\ 
  &\bm L_{1/g}^{\infty}(\Gamma):=\{\bm \mu \in L^{\infty}(\Gamma)^d : \|\bm \mu\|_{L_{1 / g}(\Gamma)}<\infty\} \text{ where }  \|\bm\eta\|_{\bm L_{1 / g}^{\infty}(\Gamma)}=\mathrm{ess}.\sup _{\Gamma} \frac{|\bm\eta|}{g}.
  \end{aligned}
\]

% %
% The trace operator $\bm \gamma$ is defined from $\bm H^1(\Omega) \to \bm \Lambda_{1/2}$. For any $\bm v \in \bm V$, we have the trace inequality
% %
% \begin{equation}\label{eq:trace}
%   \|\bm v\|_{\bm \Lambda_{1/2}} \le C_T \|\bm v\|_{H^1}.
% \end{equation}

% By the inverse trace (extension) theorem, for any $\bm\eta \in \bm \Lambda_{1/2}$, there exists a $\bm v$ satisfying $\bm v|_{\Gamma}=\bm \eta$ and 
% %
% \begin{equation}\label{eq:lift}
%   \|\bm v\|_{\bm H^1} \le C_E \|\bm \eta\|_{\bm \Lambda_{1/2}}.
% \end{equation}
% %

%%
For all $\bm u, \bm v, \bm w \in \bm V$ and $q \in Q $, we set
\[
  \begin{aligned}
    &a_0(\bm{u},\bm v) :=2\nu\sum_{i,j=1}^{d}\int_{\Omega} e_{ij}(\bm{u})e_{ij}(\bm v) ~dx, \quad e_{ij}(\bm u):=\frac{1}{2}(\frac{\partial \bm u_i}{\partial x_j}+\frac{\partial \bm u_j}{\partial x_i} ),\\
    &b(\bm v,q) :=-\int_{\Omega} \operatorname{div} \bm vq ~d x, \quad a_1(\bm{u},\bm v,\bm{w}):=\int_{\Omega} (\bm{u}\cdot\nabla)\bm v\cdot \bm{w} ~d x.
  \end{aligned} 
\]
Note that we define $a_0$, $b$ and $a_1$ on whole $\Omega$ since the integrations on the zero measure set $\Gamma$ vanish.

The bilinear forms $a_0$ and $b$ are continuous. By the Korn and Poincar\'{e} inequalities, there exist two constants $c_0$ and $c$ $(c_0 > c > 0)$ such that 
\begin{equation}\label{eq:a0}
  a_0(\bm v,\bm v)\geq c_0 \|\nabla \bm v\|_{\bm L^2}^2 \ge c \Vert \bm v \Vert^{2}_{\bm H^1} \quad \forall \bm v \in \bm V.
\end{equation}

Let $\omega \subset \mathbb R^d$ be a bounded and Lipschitz smooth domain. There is a constant $C_1$ depending on $\omega$ such that
\begin{equation}\label{eq:is1}
C_1 \|q\|_{L^2(\omega)} \le \sup_{\bm v \in \bm H^1_0(\omega)} \frac{(\nabla \cdot \bm v, q)_\omega}{ \|\bm v\|_{\bm H^1} } \quad \forall q \in L_0^2(\omega).  
\end{equation}
The inf-sup condition  ~\eqref{eq:is1} will be applied to prove the existence of pressure in $\omega =\Omega_{\mathrm{in}}$ and $\Omega_{\mathrm{out}}$ without additive constant.

For all $\bm{u},\bm v,\bm{w}\in \bm H^1(\Omega)$, we have (\cite{Temam77})
\begin{subequations}\label{eq:a1}
  \begin{align}
   & |a_1(\bm{u}, \bm v, \bm{w})| \le  C\|\bm{u}\|_{\bm L^2}^\frac{1}{2}\|\bm{u}\|_{\bm H^1}^\frac{1}{2} \|\bm v\|_{\bm H^1} \|\bm{w}\|_{\bm L^2}^\frac{1}{2}\|\bm{w}\|_{\bm H^1}^\frac{1}{2}   \quad (d=2), \label{eq:a1-2d} \\ 
   & |a_1(\bm{u}, \bm v, \bm{w})| \le  C\|\bm{u}\|_{\bm L^2}^\frac{1}{4}\|\bm{u}\|_{\bm H^1}^\frac{3}{4} \|\bm v\|_{\bm H^1} \|\bm{w}\|_{\bm L^2}^\frac{1}{4}\|\bm{w}\|_{\bm H^1}^\frac{3}{4}   \quad (d=3), \label{eq:a1-3d} \\ 
   &|a_1(\bm{u}, \bm v, \bm{w})| \le  C\|\bm{u}\|_{\bm H^1} \|\bm v\|_{\bm H^1} \|\bm{w}\|_{\bm H^1}  \quad (d=2,3). \label{a1-2d-3d}
  \end{align}
\end{subequations}
Moreover,
\begin{subequations} \label{eq:a1-1}
  \begin{align}
    &a_1(\bm{u},\bm v,\bm v)=-\frac{1}{2} \int_\Gamma \bm [u_n] |\bm v|^2 ~ds=0, \quad \forall \bm{u} \in \bm V_\sigma,~\bm v \in \bm H_0^1(\Omega), \label{eq:a1=0} \\ 
    &a_1(\bm{u},\bm v,\bm{w})=-a_1(\bm{u},\bm{w},\bm v), \quad \forall \bm{u} \in \bm V_\sigma,~\bm v,\bm{w} \in \bm H_0^1(\Omega). \label{eq:sys}
  \end{align}
\end{subequations}
%
%%%
\begin{remark}\label{rk:FIC}
For the Navier-Stokes equations under LBCF \cite{Kashiwabara13}, \eqref{eq:a1=0} is not valid, i.e., 
\[
  a_1(\bm{u},\bm v,\bm v) =\frac{1}{2} \int_{\Gamma_1} \bm u_n |\bm v|^2 ~ds \neq 0 \quad (\text{ if }\operatorname{div}\bm u=0)
\]
where $\Gamma_1$ is a subset of the boundary.
Even in the 2D case, the smallness assumption on $\bm u_0$ only admits a time-local solution.
Thus, \eqref{eq:a1=0} enables us to avoid the deficiency of energy estimates and obtain the global-in-time weak solution in 3D and strong solution in 2D.
\end{remark}
%%%

%%
\section{The well-posedness and regularity of the Stokes case}\label{sec:3}
In this section, we present the weak form of {\bf{(S-P)}} and prove the global unique existence and regularity. Throughout this section, we assume that 
\begin{equation}\label{eq:ass-S}
  \text{For a.e. } t,~\bm f(t) \in \bm V', ~g(t) \in L^2(\Gamma) \text{ with } g > 0 \text{ and } \bm u_0 \in \bm H_\sigma.
\end{equation}
Additional regularity assumptions of the data will be given later.

%%%%%%%%%%%%%
\subsection{The variational inequalities} 
We derive the parabolic variational inequalities and prove the equivalence theorem. 

Testing \eqref{eq:S-P} by $\bm v\in \bm V$, we obtain 
\[
  \int_{\Omega} \bm u' \cdot \bm v~dx - \int_{\Omega} (\nabla \cdot \bm{\sigma}) \cdot \bm v~dx = \int_{\Omega} \bm f \cdot \bm v~dx \quad \forall \bm v\in \bm V.
\]
In view of $\bm v|_{\Gamma_D} = 0$, and using the integration by parts in $\Omega_{\mathrm{in}}$ and $\Omega_{\mathrm{out}}$, respectively, we have 
\[
  \begin{aligned}
    - \int_{\Omega} (\nabla \cdot \bm{\sigma}) \cdot \bm v~dx & =\int_{\Omega_\mathrm{in}\cup \Omega_\mathrm{out}}\bm{\sigma}\cdot \nabla \bm v ~dx - \int_{\Gamma_D}\bm {\sigma n} \cdot \bm v ~ds +\int_\Gamma [\bm{\sigma n}]\cdot \bm v ~ds  \\ 
    &= \nu \int_{\Omega_\mathrm{in} \cup \Omega_\mathrm{out}}(\nabla \bm u+\nabla^\top \bm u)\cdot \nabla \bm v~dx -\int_{\Omega_\mathrm{in} \cup \Omega_\mathrm{out}}p \mathbb{I}\cdot \nabla \bm v~dx  +\int_{\Gamma}[\bm{\sigma n}]\cdot \bm v~ds \\ 
    & = a_0(\bm u, \bm v) + b(\bm v, p) + \int_{\Gamma}[\bm{\sigma n}]\cdot \bm v~ds.
  \end{aligned}
\]

A primal weak formulation of {\bf(S-P)} (i.e., \eqref{eq:S-P}, \eqref{eq:P-b}-\eqref{eq:P-e}) is as follows.
For simplicity, we omit $(t)$ when not confusing.

%%%Problem S-PDE
{\bf(S-PDE)}  For a.e $t \in (0, T)$, find $(\bm u(t),p(t)) \in \bm V_\sigma \times \mr{Q}$ such that $\bm u'(t) \in \bm V_\sigma'$, $\bm{u}(0) = \bm{u}_0$,
\begin{subequations}\label{eq:S-PDE}
\begin{align}
  & \langle \bm{u}',\bm v\rangle + a_0(\bm{u},\bm v) + b(\bm v, p) = \langle \bm{f},\bm v \rangle \quad \forall \bm v \in \bm V^0, \label{eq:S-PDE-a}\\ 
  & |[\bm\sigma \bm{n}](t)| \le g(t),~ [\bm\sigma \bm{n}](t) \cdot \bm{u} - g(t)|\bm{u}| = 0 \text{ a.e. on } \Gamma,  \label{eq:S-PDE-b}
\end{align}
\end{subequations}
where $[\bm\sigma \bm{n}](t) \in \bm \Lambda_{-1/2}$ is defined by 
\begin{equation}\label{eq:sigman-def}
\langle [\bm\sigma \bm{n}], \bm v \rangle_{\bm \Lambda_{1/2}} = \langle \bm{f},\bm v\rangle - \langle \bm{u}',\bm v\rangle - a_0(\bm u,\bm v) - b(\bm v, p) \quad \forall \bm v \in \bm V.
\end{equation}
Note that $[\bm\sigma \bm{n}](t)$ is well-defined by the duality and \eqref{eq:sigman-def} because the trace operator $\bm V \ni \bm v\rightarrow \bm v|_\Gamma$ is a surjection and every $\bm\eta \in \bm \Lambda_{1/2}$ has a continuous extension $\bm v \in \bm V$, i.e., $\bm v|_\Gamma = \bm\eta$ and $\| \bm v\|_{\bm V} \le C_E \|\bm\eta\|_{\bm \Lambda_{1/2}}$. 
It is easy to check that a classical solution of {\bf{(S-P)}} solves {\bf(S-PDE)}, and conversely, a sufficiently smooth solution of {\bf(S-PDE)} is a classical solution. 

Introduce the nonlinear functional
\[
j(t;\bm\eta)=\int_{\Gamma}g(t) |\bm\eta|~ds \quad (\forall \bm\eta \in \bm \Lambda).
\]
We will show that {\bf(S-PDE)} is equivalent to the following variational inequality. 

%%%Problem S-VI
{\bf(S-VI)$_\sigma$}  For a.e $t \in (0, T)$, find $\bm u(t) \in \bm V_\sigma$ such that $\bm u'(t) \in \bm V_\sigma'$, $\bm u(0) = \bm u_0$ and 
\begin{equation}\label{eq:S-VI-sig}
\langle \bm u',\bm v- \bm u\rangle +a_0(\bm u,\bm v-\bm u)+j(t;\bm v)-j(t;\bm u) \geq \langle \bm f,\bm v-\bm u\rangle \quad \forall \bm v  \in \bm V_\sigma.
\end{equation}
%

%%%% equivalence thm PDE-VI
\begin{theorem}\label{th:equi-S}
  The solution of {\bf(S-PDE)} solves {\bf(S-VI)$_\sigma$}. 
  Conversely, if {\bf(S-VI)$_\sigma$} admits a solution $\bm u(t)$, then there exists at least one $p(t) \in \mr{Q}$ such that $(\bm u(t),p(t))$ solves  {\bf(S-PDE)}.
  If another $p(t)^* \in \mr{Q}$ satisfies the same condition, then there exists a unique $\delta(t) \in \mathbb R$ such that $[p(t)]=[p^*(t)] +\delta(t)$. 
  Furthermore, if $\bm{u}\cdot \bm{n}(t) \neq 0$ on $\Gamma$, then $\delta(t) = 0$, i.e., the associated pressure $p(t)$ is uniquely determined.
\end{theorem}
%%%%

%%%%
\begin{proof}
%%% PDE->VI
Let $(\bm{u}, p)$ solve {\bf(S-PDE)}. \eqref{eq:S-VI-sig} follows directly from~\eqref{eq:S-PDE-b}, \eqref{eq:sigman-def}. Hence, $\bm{u}$ solves {\bf(S-VI)$_\sigma$}. 

%%% VI->PDE
Conversely, let $\bm{u}$ be the solution of {\bf(S-VI)$_\sigma$}. 
Substituting $\bm v = \bm u \pm \bm w$ with arbitrary $\bm w \in \bm V_\sigma^0$ into \eqref{eq:S-VI-sig} yields  
\[
\langle \bm u',\bm{w}\rangle + a_0(\bm u,\bm{w}) = \langle \bm f,\bm{w}\rangle \quad \forall \bm{w} \in \bm V^0_\sigma. 
\]
By the inf-sup condition \eqref{eq:is1}, there exists a unique $\mr{p}(t) = (\mr{p}_\mathrm{in}, \mr{p}_\mathrm{out}) \in L_0^2(\Omega_\mathrm{in}) \times L_0^2(\Omega_\mathrm{out})$ such that 
\[
\langle \bm u',\bm{w}\rangle + a_0(\bm u,\bm{w}) + b(\bm{w}, \mr{p}) = \langle \bm f, \bm{w}\rangle \quad \forall \bm{w} \in \bm V^0. 
\]
There exists a unique $\mr{\bm\lambda} \in \bm\Lambda_{-1/2}$ such that 
\begin{equation}\label{eq:equi-S-2}
\langle \bm u', \bm v\rangle + a_0(\bm{u},\bm v) + b(\bm v, \mr{p}) + \langle \mr{\bm\lambda},\bm v\rangle_{\bm \Lambda_{1/2}} = \langle \bm f,\bm v\rangle \quad \forall \bm v \in \bm V. 
\end{equation}
For all $\bm{w} \in \bm V_\sigma$, it follows from \eqref{eq:S-VI-sig} and \eqref{eq:equi-S-2} that 
\[
\langle \mr{\bm\lambda}, \bm{w} - \bm u \rangle_{\bm \Lambda_{1/2}} \le j(t;\bm{w}) - j(t;\bm u) \quad \forall \bm{w} \in \bm V_\sigma.
\]
Taking $\bm{w} = 0, 2\bm{u}$ into above inequalities results in
\[
\langle \mr{\bm\lambda}, \bm u \rangle_{\bm \Lambda_{1/2}} = j(t;\bm u), \quad \langle \mr{\bm\lambda}, \bm{w} \rangle_{\bm \Lambda_{1/2}} \le j(t; \bm{w})  \quad \forall \bm{w} \in \bm V_\sigma, 
\]
or equivalently, 
\[
\langle \mr{\bm\lambda}, \bm\mu \rangle_{\bm \Lambda_{1/2}} \le j(t;\bm\mu) = \int_\Gamma g(t)|\bm\mu|~ds \quad \forall \bm\mu \in \mr{\bm \Lambda}_{1/2},   
\]
which means that $\mr{\bm\lambda}$ is a continuous linear functional on $\mr{\bm \Lambda}_{1/2}$. 
Using the Hahn-Banach theorem, we can extend $\mr{\bm\lambda}$ to a bounded linear functional $\bm\lambda$ on $\bm L^1_g(\Gamma)$. This functional satisfies
\[
  |\langle \bm\lambda, \bm\mu \rangle_{\bm L_g^1(\Gamma)}| \le \int_\Gamma g(t)|\bm\mu| ~ds=\|\bm\mu \|_{\bm L_g^1(\Gamma)} \quad \forall \bm\mu \in \bm L_g^1(\Gamma),
\]
which implies $\bm\lambda \in ( \bm L_g^1(\Gamma) )' = \bm L_{1/g}^\infty(\Gamma)$ and $|\bm\lambda| \le g(t)$ a.e. on $\Gamma$. In addition, we have 
\[
\int_\Gamma \bm\lambda \cdot \bm{u}~ds = \int_\Gamma\mr{\bm\lambda} \cdot \bm{u}~ds =  j(t;\bm u) = \int_\Gamma g(t)|\bm{u}|~ds, 
\]
together with $|\bm\lambda| \le g(t)$ a.e., we conclude $\bm\lambda \cdot \bm u = g(t)|\bm u|$ a.e. on $\Gamma$. 

Notice that such an extension may not be unique. In fact, since 
\[
\langle \mr{\bm\lambda} -\bm\lambda, \bm\mu \rangle_{\bm \Lambda_{1/2}} =0 \quad  \text{for all } \bm\mu \in \mr{\bm \Lambda}_{1/2} \quad  (\text{note that} \int_\Gamma \bm\mu \cdot \bm n ~ds =0),
\]
there exists a constant $\delta(t)$ such that $\mr{\bm\lambda} - \bm\lambda = \delta(t) \bm{n}$. 
If $\delta(t)$ is given, then we can find two constants $k_\mathrm{in}(t)$ and $k_\mathrm{out}(t)$ satisfying 
\[
  k_\mathrm{out}(t) - k_\mathrm{in}(t) = \delta(t), \quad  (k_\mathrm{in}(t), 1)_{\Omega_\mathrm{in}} + (k_\mathrm{out}(t),1)_{\Omega_\mathrm{out}} = 0.
\]
Note that $k_\mathrm{in}(t)$ and $k_\mathrm{out}(t)$ are uniquely determined by $\delta(t)$. 
Now, we set $p(t) := (\mr{p}_\mathrm{in} + k_\mathrm{in}(t), \mr{p}_\mathrm{out} + k_\mathrm{out}(t)) \in \mr{Q}$ and $[\bm\sigma \bm{n}](t):=\bm\lambda = \mr{\bm\lambda} - \delta(t) \bm{n}$. 
It is easy to verify that \eqref{eq:equi-S-2} still holds if $(\mr{\bm\lambda}, \mr{p})$ is replaced with $([\bm\sigma \bm{n}], p)$, i.e., 
\[
  \langle \bm u', \bm v\rangle + a_0(\bm u,\bm v) + b(\bm v, p) + \langle [\bm\sigma \bm{n}],\bm v\rangle_{\bm \Lambda_{1/2}} = \langle \bm{f},\bm v\rangle \quad \forall \bm v \in \bm V. 
\]
Hence, $(\bm{u} , p)$ is a solution of {\bf(S-PDE)}. 

If $\bm{u}\cdot \bm{n} = 0$ on $\Gamma$, then for any $\delta(t) \in \mathbb{R}$ such that $|[\bm\sigma \bm{n}]| =|\mr{\bm\lambda} - \delta(t)\bm n| \le g(t)$ a.e. on $\Gamma$, 
$(\bm{u},p)$ solves {\bf(S-PDE)}.   
Note that $ [p] = [\mr{p}] + k_\mathrm{out} - k_\mathrm{in} =  [\mr{p}] + \delta(t)$.
Therefore, if both $(\bm{u} , p)$ and $(\bm{u} , p^*)$ are the solutions of {\bf(S-PDE)}, then $[p]-[p^*]$ must be a constant.  

When $\{x \in \Gamma: \bm{u} \cdot \bm{n} \neq 0\}$ has a positive measure, since $\int_\Gamma \bm u\cdot \bm n~ds =0 (\text{ by } \nabla \cdot \bm u =0)$, then the following two subsets must have positive measure:
\[
A_+:=\{x\in \Gamma : \bm{u} \cdot \bm{n} >0\},\quad A_-:=\{x\in \Gamma : \bm{u} \cdot \bm{n} < 0\}.
\] 
If there is another $\bar{\delta}(t)$ such that 
\[ 
|\mr{\bm\lambda} - \bar{\delta}(t)\bm n| \le g(t) \text{ and }  (\mr{\bm\lambda} - \bar{\delta}(t)\bm n)\cdot \bm u = g(t)|\bm u|,
\]
together with  $(\mr{\bm\lambda} - \delta(t)\bm n)\cdot \bm u = g(t)|\bm u|$, we see that $(\delta(t)-\bar{\delta}(t))\bm n\cdot \bm u = 0$, which implies $\delta(t) = \bar{\delta}(t) $.
Hence, $\delta(t)$ is unique, so do $\bm \lambda$ and $p$.
\end{proof}
%%%%

%%%%
\begin{remark} 
  According to $|\mr{\bm\lambda} - \delta(t)\bm n| \le g(t)$, together with the triangle inequality and $|\mr{\bm\lambda}| \le g(t)$,  
  the constant $\delta(t)$ is limited to the closed interval $[-2g(t),2g(t)]$. 
\end{remark}
%%%%

Given $\bm f \in L^2(0, T; \bm V')$ and $g \in L^2(0,T; L^2(\Gamma))$, we will consider the following weak formulation. 

%%% Problem S-VI-weak
{\bf(S-VI)$_\sigma^w$}  Find $\bm u \in L^2(0, T; \bm V_\sigma)$ such that $\bm u(0) = \bm u_0$ and 
\begin{equation}\label{eq:S-VI-sigw}
  \frac{d}{dt}(\bm u,\bm{v}) - \frac{1}{2}\frac{d}{dt}(\bm u,\bm u) +a_0(\bm u,\bm v-\bm u)  + j(t;\bm{v})-j(t;\bm u) \geq \langle \bm f,\bm{v}-\bm u \rangle \quad \forall \bm v  \in \bm V_\sigma.
\end{equation}
We will see that $\bm{u}' \in L^2(0, T; \bm V_\sigma')$ actually holds true for {\bf(S-VI)$_\sigma^w$}. Thus, \eqref{eq:S-VI-sigw} can be equivalently stated as follows: 
\begin{equation}\label{eq:S-VI-sigw'}
  \int_0^T \langle \bm u',\bm v-\bm u \rangle + a_0(\bm u,\bm v-\bm u) + j(t;\bm{v})-j(t;\bm u)~dt \geq \int_0^T\langle \bm f,\bm{v}-\bm u \rangle~dt \quad \forall \bm{v}  \in L^2(0, T; \bm V_\sigma).
\end{equation}
%

%%%%%%%%%%%
\subsection{The regularization problem}
We shall apply the regularization method to establish the well-posedness theory of {\bf(S-VI)$_\sigma$}. 
The essential idea is to approximate the non-differentiable nonlinear term $j(t;\cdot)$ by a Frechet differentiable functional $j_\ep(t;\cdot)$ with a small positive parameter $\ep$. 
We introduce the regularization function $\rho_\ep(\cdot)$ satisfying: 
%%%
\begin{enumerate}
  \item[(1)] $\rho_\ep \in C^2(\mathbb R^d)$, $\rho_\ep$ is non-negative, convex function and 
  \begin{equation}\label{eq:rho-err}
    |\rho_\ep(\bm z)-|\bm z| |\leq \ep \quad \forall \bm z \in \mathbb{R}^d;
  \end{equation} 
  \item[(2)] $\alpha_\ep(\bm z)$ denotes $\nabla\rho_\ep(\bm z)$, $|\alpha_\ep(\bm z)| \leq 1$ and
  \begin{equation}\label{eq:alpha_ep}
  \alpha_\ep(\bm z) \cdot \bm{z} \geq 0 \quad \forall \bm z \in \mathbb{R}^d.
  \end{equation}
  Moreover, the Hessian matrix of $\rho_\ep$, denoted by $\beta_\ep$ is semi-positive definite, i.e.,
  \begin{equation}\label{eq:beta_ep}
  \bm{y}^\top \beta_\ep(\bm z) \bm{y} \geq 0 \quad \forall \bm y, \bm z \in \mathbb{R}^d.
  \end{equation}
\end{enumerate}
%%%
For instance, we can choose $\rho_\ep(\bm z) =\sqrt{|\bm z|^2+\ep^2} - \ep$.

Now, we define the nonlinear regularization functional
\[
j_\ep(t;\bm\eta) := \int_\Gamma g(t) \rho_\ep ( \bm\eta )~ds \qquad \forall \bm\eta \in \bm \Lambda_{1/2}. 
\]
Now, replacing $j(t; \cdot)$ with $j_\ep(t; \cdot)$ in the variational inequality {\bf(S-VI)$_\sigma$}, we get the following regularization problem. 

%%% Problem S-VI-ep
{\bf(S-VI)$_{\sigma,\ep}$}  For a.e. $t\in (0, T)$, find $\bm u_\ep(t) \in \bm V_\sigma$ such that $\bm u_\ep'(t) \in \bm V_\sigma'$, $\bm u_\ep(0) = \bm u_0$ and
\begin{equation}\label{eq:S-VIe}
  \langle \bm u'_\ep, \bm{v}-\bm u_\ep\rangle +a_0(\bm u_\ep,\bm{v}-\bm u_\ep)+j_\ep(t;\bm{v})-j(t;\bm u_\ep) \geq \langle \bm f,\bm{v}-\bm u_\ep \rangle \quad \forall \bm{v}  \in \bm V_\sigma.
\end{equation}
Since $\rho_\ep$ is differentiable, the above variational inequality is actually equivalent to the following nonlinear variational problem. 

%%% Problem S-VE-ep
{\bf(S-VE)$_{\sigma,\ep}$}  For a.e. $t\in (0, T)$, find $\bm u_\ep(t) \in \bm V_\sigma$ such that $\bm u'_\ep(t) \in \bm V_\sigma'$, $\bm u_\ep(0) = \bm u_0$ and   
\begin{equation}\label{eq:S-VEe}
  \langle \bm u'_\ep,\bm{v} \rangle +a_0(\bm u_\ep,\bm{v})+\int_{\Gamma}g \alpha_\ep(\bm u_\ep)\cdot \bm{v}~ds =\langle \bm f, \bm{v}\rangle  \quad \forall \bm{v}  \in \bm V_\sigma.
\end{equation}
%

%%%
\begin{remark}\label{rk:equi-S}
  The PDE of the regularization problem {\bf(S-VE)$_{\sigma,\ep}$} is to replace the subgradient-type interface condition \eqref{eq:subd} with the Robin type condition
  \[
  [\bm\sigma(\bm u_\ep, p_\ep) \bm n] = g\alpha_\ep(\bm u_\ep) \text{ on } \Gamma. 
  \]
\end{remark}
%%%

We also consider the following weak formulation which is equivalent to {\bf(S-VE)$_{\sigma,\ep}$} with $\bm{u}_\ep \in L^2(0, T; \bm V_\sigma)$ and $\bm{u}_\ep' \in L^2(0, T; \bm V_\sigma')$.

%%%
{\bf(S-VE)$_{\sigma,\ep}^w$}  Find $\bm{u}_\ep \in L^2(0, T; \bm V_\sigma)$ such that $\bm{u}_\ep(0) = \bm u_0$ and 
\[
  \frac{d}{dt}(\bm{u}_\ep,\bm{v}) +a_0(\bm u_\ep,\bm v) + \int_{\Gamma}g\alpha_\ep(\bm u_\ep)\cdot \bm{v}~ds = \langle \bm f,\bm{v} \rangle \quad \forall \bm{v}  \in \bm V_\sigma.
\]
%

%%%%%%%%%%%%%%%
\subsection{The unique existence of weak solution}
First, we demonstrate the existence of the regularization problem {\bf(S-VE)$_{\sigma,\ep}^w$} using the Galerkin method. 
Then, by passing to the limit $\ep \rightarrow 0$, we show the well-posedness of {\bf(S-VI)$_\sigma^w$}. 

%%% thm VE
\begin{theorem}\label{th:S-wp-ue}
Under the assumptions of \eqref{eq:ass-S}, there exists a unique weak solution $\bm u_\ep \in L^2(0, T; \bm V_\sigma )\cap L^\infty(0, T; \bm H_\sigma)$ of {\bf(S-VE)$_{\sigma,\ep}^w$} with the estimates
\begin{subequations}\label{eq:S-weak-e}
\begin{align}
  & \|\bm u_\ep\|_{L^\infty(0, T; \bm H_\sigma)}
  + \| \bm u_\ep \|_{L^2(0, T; \bm  V_\sigma)} \le C(\|\bm f\|_{L^2(0, T; \bm V')} + \|\bm u_0\|), \label{eq:S-weak-e-a} \\
  & \| \bm u'_\ep\|_{L^2(0, T; \bm V_\sigma')} \le C(\|\bm f\|_{L^2(0, T; \bm V')} + \|g\|_{L^2(0, T; L^2(\Gamma))} + \|\bm u_0\|). \label{eq:S-weak-e-b}
\end{align}
\end{subequations}
\end{theorem}
%%%

%%%
\begin{proof}
The proof is divided into five steps. In (Step 1), we introduce the Galerkin approximation problem.  
The a-priori estimates of the approximated solutions are obtained in (Step 2). In (Step 3), we pass to the limit and show the existence of the weak solution.
(Step 4) is to demonstrate that $\bm{u}'_{\ep} \in L^2(0, T; \bm V_\sigma')$ and $\bm{u}_\ep(0) = \bm{u}_0$. We show the uniqueness in (Step 5). 

%% Step1
(Step 1) {\bf Galerkin's approximation problem}. 
Let $\{\bm{w}_k\}_{k=1}^\infty$ be a complete basis of the separable Hilbert space $\bm V_\sigma$, and define $\bm V^m_\sigma := \mathrm{span}\{\bm{w}_k\}_{k=1}^m$. 
For each $m $, we find the solution $\bm u_\ep^m= \sum_{i=1}^m c_k(t) \bm{w}_k \in \bm V^m_\sigma$ of the Galerkin approximation problem: 
\begin{subequations}\label{eq:S-VEem}
    \begin{align}
    (\bm u_\ep^{m \prime}, \bm{w}_k)+a_0(\bm u_\ep^m, \bm{w}_k)+\int_{\Gamma}g \alpha_\ep(\bm u_\ep^m)\cdot \bm{w}_k ~ds &= \langle \bm f, \bm{w}_k \rangle, \label{eq:S-VEem-a}\\
    \bm u_\ep^m(0)&= \bm u_0^m, \label{eq:S-VEem-b}
    \end{align}
\end{subequations}
where $\bm{u}_0^m := P_m \bm{u}_0$ and $P_m$ is the $L^2$-projection of $\bm H_\sigma$ onto $\bm V^m_\sigma$ satisfying  
\begin{equation}\label{eq:S-u0}
\|\bm{u}_0^m\| \le \|\bm{u}_0\|, \quad  \bm{u}_0^m \to \bm{u}_0 \text{ in } \bm H_\sigma.
\end{equation} 
Notice that \eqref{eq:S-VEem} is a nonlinear ODE system of $\{c_k(t)\}_{k=1}^m$. 
By the standard theory of ODEs, there exists a $T_1>0$ such that \eqref{eq:S-VEem} admits unique solutions $c_k(t)\in C^2([0, T_1])~(k=1,\ldots,m)$. 
The a-priori estimate in the next step indicates that $T_1$ can be replaced by $T$.  

%%% Step2
(Step 2) {\bf A-priori estimate}.
Multiplying \eqref{eq:S-VEem-a} by $c_k(t)$ and summing for $k=1,\ldots m$, we obtain
\[
(\bm{u}_\ep^{m\prime},\bm{u}_\ep^m)+a_0(\bm{u}_\ep^m,\bm{u}_\ep^m) +\int_\Gamma g\alpha_\ep(\bm{u}_\ep^m)\cdot\bm{u}_\ep^m ~ds=\langle \bm f,\bm{u}_\ep^m \rangle, 
\]
which yields (by \eqref{eq:a0} and \eqref{eq:alpha_ep})
\begin{equation}\label{eq:S-est-0}
\frac{1}{2}\frac{d}{dt}\|\bm{u}_\ep^m\| + c \|\bm{u}_\ep^m\|_{\bm H^1}^2 
\leq \langle \bm f,\bm{u}_\ep^m \rangle \leq \frac{1}{2c} \|\bm f\|_{\bm V'}^2 + \frac{c}{2} \|\bm{u}_\ep^m\|_{\bm V_\sigma}^2.
\end{equation}
Integrating the above equation w.r.t. $t$, we obtain the estimate
\begin{equation}\label{eq:S-est-I}
  \sup_{t \in [0, T]}\|\bm{u}_\ep^m(t)\|^2 + c \|\bm{u}_\ep^m \|_{L^2(0, T; \bm V_\sigma)}^2 
  \leq C \| \bm f \|_{L^2(0, T; \bm V')}^2 + \|\bm{u}_0\|^2,    
\end{equation}
which implies the $L^\infty(0, T; \bm L^2)$ and $L^2(0, T; \bm H^1)$ norms of $\bm{u}_\ep^m$ are bounded independent of $\ep$ and $m$. 

%%% Step3
(Step 3) {\bf Passing to the limit for $m \to \infty$}.
Since $\bigcup_{m=1}^\infty \bm V^m_\sigma$ is dense in $\bm V_\sigma$, we can find a subsequence of $\{\bm{u}_\ep^m\}_m$ (denoted by ${\bm{u}_\ep^m}$ for briefness) such that, as $m \to \infty$,  
\begin{equation}\label{eq:S-est-2}
  \begin{aligned}
    &\bm{u}_{\ep}^m  \rightharpoonup \bm{u}_{\ep} \quad & \text{ in } L^2(0, T; \bm V_\sigma), \\
    &\bm{u}_{\ep}^m \stackrel{*}{\rightharpoonup} \bm{u}_{\ep} \quad & \text{ in } L^\infty(0, T; \bm H_\sigma). 
  \end{aligned}
\end{equation}
The estimate \eqref{eq:S-weak-e-a} follows from \eqref{eq:S-est-I}. 
By the trace theorem (\cite[Theorem II.6.2]{Necas12}), we have 
\[
  \bm{u}_{\ep}^m \rightarrow \bm{u}_{\ep} \quad\ text{ in } L^2((0, T)\times \Gamma)^d, 
\]
which implies $\bm{u}_{\ep}^m \rightarrow \bm{u}_{\ep}$ a.e. on $(0, T) \times \Gamma$. Together with the continuity of $\alpha_\ep(\cdot)$, we have
\begin{equation}\label{eq:alpha}
\alpha_\ep(\bm{u}_{\ep}^m) \rightarrow \alpha_\ep(\bm{u}_{\ep}) \text{ a.e. on } (0, T) \times \Gamma.
\end{equation}

For any $\psi \in C^1([0, T])$ with $\psi(T)=0$, multiplying \eqref{eq:S-VEem-a} by $\psi(t)$ and using the integration by parts, we get
\[
  \begin{aligned}
    &- \int_0^T(\bm{u}_\ep^m,\psi'(t)\bm{w}_k)~dt - (\bm{u}_0^m,\psi(0) \bm{w}_k)+\int_0^Ta_0(\bm{u}_\ep^m ,\psi(t)\bm{w}_k)~dt  \\ 
    &+ \int_0^T\int_\Gamma g\alpha_\ep(\bm{u}_\ep^m)\cdot\psi(t)\bm{w}_k ~ds dt =\int_0^T\langle \bm{f},\psi(t)\bm{w}_k \rangle ~dt.
\end{aligned}
\]

Passing to the limit for $m \rightarrow \infty$, using \eqref{eq:S-u0}, \eqref{eq:S-est-2}, ~\eqref{eq:alpha} and Lebesgue's dominated convergence theorem, we find the limit 
\[
  \begin{aligned}
    &-\int_0^T(\bm{u}_\ep,\psi'(t)\bm{w}_k)~dt - (\bm{u}_0,\psi(0) \bm{w}_k)+\int_0^Ta_0(\bm{u}_\ep ,\psi(t)\bm{w}_k)~dt  \\ 
    &+\int_0^T\int_\Gamma g\alpha_\ep(\bm{u}_\ep)\cdot\psi(t)\bm{w}_k ~ds dt =\int_0^T\langle \bm{f},\psi(t)\bm{w}_k \rangle ~dt.  
\end{aligned}
\]
Since $\bigcup_{m=1}^\infty \bm V^m_\sigma$ is dense in $\bm V_\sigma$, we have 
\begin{equation}\label{eq:S-est-3}
  \begin{aligned}
    &-\int_0^T(\bm{u}_\ep,\bm{v})\psi'(t)~dt - (\bm{u}_0, \bm{v})\psi(0) +\int_0^Ta_0(\bm{u}_\ep ,\bm{v}) \psi(t)~dt  \\ 
    &+\int_0^T\int_\Gamma g\alpha_\ep(\bm{u}_\ep)\cdot\bm{v}\psi(t) ~ds dt =\int_0^T\langle \bm{f},\bm{v}\rangle  \psi(t)~dt \quad \forall \bm{v} \in \bm V_\sigma.  
\end{aligned}
\end{equation}
%

%%% Step4
(Step 4) According to \cite[Lemma III.1.1]{Temam77}, we define $\bm{u}_\ep'$ in the weak sense: for any $t \in (0, T]$ and $\psi \in C_0^\infty((0, T))$, 
\[
\int_0^T  \langle \bm{u}_\ep'(t), \bm{v} \rangle \psi(t) ~dt = \int_0^T \frac{d}{dt} ( \bm{u}_\ep(t), \bm{v} ) \psi(t) ~dt -\int_0^T (\bm{u}_\ep(t),\bm{v})\psi'(t)~dt \quad \forall \bm v \in \bm V_\sigma. 
\]
Therefore, we have $\bm{u}_\ep' \in L^2(0, T; \bm V_\sigma')$ from 
\begin{equation}\label{eq:S-est-4}
  \begin{aligned}
    & \int_0^T \big(\langle \bm{u}_\ep',\bm{v}\rangle + a_0(\bm{u}_\ep , \bm{v}) + \int_\Gamma g \alpha_\ep(\bm{u}_\ep)\cdot\bm{v} - \langle \bm{f}, \bm{v}\rangle \big) \psi(t)~dt = 0 \quad \forall \bm{v} \in \bm V_\sigma,   
\end{aligned}
\end{equation}
which is \eqref{eq:S-VEe} in the weak sense. The estimate \eqref{eq:S-weak-e-b} follows from \eqref{eq:S-est-4} and \eqref{eq:S-weak-e-a}. 

It remains to verify the initial condition $\bm{u}_\ep(0)=\bm{u}_0$. 
Multiplying \eqref{eq:S-VEe} by $\psi \in C^1([0, T])$ with $\psi(T) = 0$ and integrating over $[0, T]$, together with integration by parts, we get 
\begin{equation}\label{eq:S-est-5}
  \begin{aligned}
    &-\int_0^T(\bm{u}_\ep,\bm{v}) \psi'(t)~dt - (\bm{u}_\ep(0),\bm v) \psi(0)+\int_0^Ta_0(\bm{u}_\ep ,\bm{v})\psi(t) ~dt  \\ 
    &+\int_0^T\int_\Gamma g\alpha_\ep(\bm{u}_\ep)\cdot\bm{v}\psi(t) ~dx dt =\int_0^T\langle \bm f,\bm{v} \rangle \psi(t)~dt  \quad \forall \bm{v} \in \bm V_\sigma.
  \end{aligned}
\end{equation}
Subtracting \eqref{eq:S-est-5} from \eqref{eq:S-est-3}, we get
\[ 
(\bm{u}_\ep(0)-\bm{u}_0,\bm{v}) \psi(0)=0 \quad  \forall \bm{v}\in \bm V_\sigma, 
\]
which says $\bm{u}_\ep(0)=\bm{u}_0$. 
Hence, $\bm{u}_\ep$ is a weak solution of {\bf(S-VE)$_{\sigma,\ep}^w$}. 

%%% Step5
(Step 5) {\bf Uniqueness.} Suppose there are two solutions $\bm{u}_\ep^1$ and $\bm{u}_\ep^2$ of {\bf(S-VE)$_{\sigma,\ep}^w$}. We find that: $\bm{u}^1_\ep(0) - \bm{u}^2_\ep(0) = 0$, and 
\[
  \frac{d}{dt} (\bm{u}^1_\ep - \bm{u}^2_\ep, \bm{v}) +a_0(\bm{u}^1_\ep - \bm{u}^2_\ep ,\bm{v})+\int_{\Gamma}g (\alpha_\ep(\bm{u}^1_\ep) - \alpha_\ep(\bm{u}^2_\ep))\cdot \bm{v}~ds = 0 \quad \forall \bm{v}  \in \bm V_\sigma.
\]
Substituting $\bm{v} = \bm{u}^1_\ep - \bm{u}^2_\ep$ into the above equation and integrating over $[0, T]$, we observe that (by using \eqref{eq:a0} and the convexity of $\alpha_\ep(\bm{z})$): 
\[
\frac{1}{2}\big( \|\bm{u}^1_\ep(t) - \bm{u}^2_\ep(t)\|_{L^2}^2 - \|\bm{u}^1_\ep(0) - \bm{u}^2_\ep(0)\|_{L^2}^2 \big) + \int_0^t c \| \bm{u}^1_\ep(\tau) - \bm{u}^2_\ep(\tau) \|_{H^1}^2~d\tau  \le 0,  
\]
which yields $\bm{u}^1_\ep = \bm{u}^2_\ep$. 
\end{proof}
%%%

%%%%
Next, we demonstrate the existence of {\bf(S-VI)$_\sigma^w$} by the passage to the limit for $\ep \rightarrow 0$. 
%%% thm VI
\begin{theorem}\label{th:S-wp-u}
Under the assumptions of \eqref{eq:ass-S}, there exists a unique weak solution $\bm u \in L^2(0, T; \bm V_\sigma )\cap L^\infty(0, T; \bm H_\sigma)$ of {\bf(S-VI)$_\sigma^w$} with the estimate 
\begin{subequations}\label{eq:S-weak}
\begin{align}
& \| \bm u \|_{L^\infty(0, T; \bm H_\sigma)} + \| \bm u \|_{L^2(0, T; \bm V_\sigma)} \le C(\|\bm f\|_{L^2(0, T; \bm V')} + \|\bm u_0\|), \label{eq:S-weak-a} \\
& \| \bm u' \|_{L^2(0, T; \bm V_\sigma')} \le C(\|\bm f\|_{L^2(0, T; \bm V')} + \|g\|_{L^2(0, T; L^2(\Gamma))} + \|\bm u_0\|) \label{eq:S-weak-b}
\end{align}
\end{subequations}
\end{theorem}
%%%

%%%
\begin{proof}
Since \eqref{eq:S-weak-e}, there exists a subsequence $\{ \bm{u}_\ep \}_\ep$ such that, as $\ep \rightarrow 0$,  
\begin{subequations}\label{eq:prf-S-wp-1}
  \begin{align}
    &\bm{u}_{\ep} \rightharpoonup  \bm{u}\quad & \text{ in } L^2(0, T; \bm V_\sigma), \label{eq:prf-S-wp-1-a}\\
    &\bm{u}_{\ep} \stackrel{*}{\rightharpoonup}  \bm{u}\quad & \text{ in } L^\infty(0, T; \bm H_\sigma), \label{eq:prf-S-wp-1-b}\\
    &\bm{u}'_{\ep} \rightharpoonup  \bm{u}'\quad & \text{ in } L^2 (0, T; \bm V_\sigma'), \label{eq:prf-S-wp-1-c}
  \end{align}
\end{subequations}
and $\bm{u}(0) = \bm{u}_0$ because of $\bm{u}_\ep(0) = \bm{u}_0$. The estimate \eqref{eq:S-weak} follows directly from \eqref{eq:S-weak-e}. 
We still have 
\begin{equation}
\bm{u}_{\ep} \rightarrow \bm{u} \quad \text{ in }  L^2((0, T)\times\Gamma)^d. \label{eq:L2} 
\end{equation}
Thus we have $\bm{u}_{\ep} \rightarrow \bm{u}$ a.e. on $(0, T)\times\Gamma$. Furthermore, by the continuity of $\rho_\ep(\bm{z})$ and \eqref{eq:rho-err}, we have
\[
\rho_\ep(\bm{u}_{\ep}) \rightarrow |\bm{u}| \text{ a.e. on } (0, T)\times\Gamma. 
\]

For any $\bm{v} \in L^2(0, T; \bm V_\sigma)$, integrating \eqref{eq:S-VIe} over $[0, T]$, we get
\begin{equation}\label{eq:S-est-7}
  \int_0^T\big( \langle \bm{u}_\ep', \bm{v}-\bm{u}_\ep\rangle  + a_0(\bm{u}_\ep ,\bm{v}-\bm{u}_\ep) + j_\ep (t;\bm{v})-j_\ep (t;\bm{u}_\ep )\big)~dt \geq \int_0^T \langle \bm{f},\bm{v}-\bm{u}_\ep\rangle~dt. 
\end{equation}
The first term of the left-hand side is decomposed into three terms:
\[
\begin{aligned}
  \int_0^T \langle \bm{u}_\ep', \bm{v}-\bm{u}_\ep\rangle ~dt = & \int_0^T  \langle \bm{u}_\ep' - \bm{u}', \bm{u}-\bm{u}_\ep \rangle ~dt  + \int_0^T \langle \bm{u}_\ep' - \bm{u}', \bm{v}-\bm{u}\rangle~dt \\ 
  & + \int_0^T\langle \bm{u}', \bm{v}-\bm{u}_\ep\rangle ~dt =: I_{1\ep} + I_{2\ep} + I_{3\ep}. 
\end{aligned}
\]
By \eqref{eq:prf-S-wp-1-a} and \eqref{eq:prf-S-wp-1-c}, we have 
\[
\lim_{\ep\rightarrow 0} I_{2\ep} = 0, \quad \lim_{\ep\rightarrow 0}I_{3\ep} = \int_0^T \langle \bm{u}', \bm{v}-\bm{u}\rangle ~dt.
\]
In view of $ I_{1\ep} = -\frac{1}{2} \| \bm{u}(T)-\bm{u}_\ep(T) \|_{\bm L^2}^2 \le 0$, we get
\[
\lim_{\ep\rightarrow 0} \int_0^T \langle \bm{u}_\ep' , \bm{v}-\bm{u}_\ep\rangle ~dt \le \int_0^T \langle \bm{u}', \bm{v}-\bm{u}\rangle ~dt.
\]
The second term on the left-hand side of \eqref{eq:S-est-7} is calculated as follows:
\[
\begin{aligned}
\lim_{\ep\rightarrow 0} \int_0^T a_0(\bm{u}_\ep ,\bm{v}-\bm{u}_\ep)~dt & =  \lim_{\ep\rightarrow 0} \int_0^T a_0(\bm{u}_\ep-\bm{u} ,\bm{u}-\bm{u}_\ep) + a_0(\bm{u}_\ep-\bm{u} ,\bm{v}-\bm{u}) + a_0(\bm{u} ,\bm{v}-\bm{u}_\ep)~dt \\
& \le  \int_0^T a_0(\bm{u} ,\bm{v}-\bm{u})~dt \quad (\text{ by } a_0(\bm{u}_\ep-\bm{u} ,\bm{u}-\bm{u}_\ep) \le 0 \text{ and }\eqref{eq:prf-S-wp-1-a})
\end{aligned}
\]
By \eqref{eq:rho-err} and \eqref{eq:L2}, we see that, as $\ep \to 0$,
\[
\begin{aligned}
  &\left| \int_0^T j_\ep(t; \bm{v}) - j(t; \bm{v})~dt \right| \le \int_0^T \int_\Gamma g\big| \rho_\ep(\bm{v}) - |\bm{v}|\big|~dsdt \le \ep \|g\|_{L^1(0, T;L^1(\Gamma))} \to 0,\\ 
  & \begin{aligned} \left|\int_0^T j_\ep(t; \bm{u}_\ep) - j(t;\bm{u}) ~dt \right| & \le  \int_0^T |j_\ep(t;\bm{u}_\ep) - j(t;\bm{u}_\ep)| ~dt + \int_0^T \int_\Gamma \big|g(|\bm{u}_\ep| - |\bm{u}|)\big| ~dsdt\\ 
  & \le \ep \|g\|_{L^1(0, T;L^1(\Gamma))}  + \|g\|_{L^2(0, T;L^2(\Gamma))}\|\bm{u}_\ep -\bm{u}\|_{L^2(0, T; \bm \Lambda)}  \to 0.
  \end{aligned} 
\end{aligned}
\]  
Then, by the triangle inequality, we have 
\[
\lim_{\ep\rightarrow 0} \int_0^T j_\ep (t; \bm{v})-j_\ep (t; \bm{u}_\ep )~dt = \int_0^T j (t; \bm{v})-j (t; \bm{u} )~dt. 
\]
Hence, the limit as $\ep \to 0$ in \eqref{eq:S-est-7} combined with Lebesgue's convergence theorem yields \eqref{eq:S-VI-sigw'}.
\end{proof}
%%%%

% 添加remark说明 u=0 on Gamma
%%%
\begin{remark}\label{rk:S-Dirichlet}
Given sufficiently smooth $\bm f$ and $ \bm u_0$, we can find a large enough threshold function $g$ which ensures  $\bm u = \bm 0$ on $\Gamma$. A brief proof of this result is as follows. 

First, we consider the Stokes equations  in $\Omega_{\mathrm{in}}$ and $\Omega_{\mathrm{out}}$ with the Dirichlet B.C., respectively. 
According to classical regularity results of the Dirichlet problem, the solutions denoted by $(\bar{\bm u}, \bar{\bm p})|_{\Omega_{\mathrm{in}}}$ and $(\bar{\bm u}, \bar{\bm p})|_{\Omega_{\mathrm{out}}}$ are sufficiently smooth for given $\bm f$ and $\bm u_0$ satisfying the suitable regularity assumptions.
For  $(\bar{\bm u}, \bar{\bm p})$ in $\Omega_{\mathrm{in}}$ and $\Omega_{\mathrm{out}}$, we can define stress vector on $\Gamma$, denoted by $\bar {\bm \sigma}_{\mathrm{in}}\bm n$ and $\bar {\bm \sigma}_{\mathrm{out}}\bm n$, respectively.
Now for {\bf{(S-P)}}, let's take 
\[
g > |\bar {\bm \sigma}_{\mathrm{out}}\bm n- \bar {\bm \sigma}_{\mathrm{in}}\bm n|,\quad  \text{ a.e. on } (0,T)\times \Gamma.
\] 
% Under the same $\bm f$ and $\bm u_0$, we must have the solution $\bm u$ of the FIC problem satisfies $\bm u = 0$ on $\Gamma$ because the Dirichlet solution also solves {\bf{(S-P)}}  in which uniqueness holds.
Given the same $\bm f$ and $\bm u_0$, the solution $\bm u$ of the FIC problem must satisfy $\bm u = 0$ on $\Gamma$ because the Dirichlet solution also solves {\bf{(S-P)}} where uniqueness holds.
\end{remark}
%%%

%%%% S-Strong sol.
\begin{theorem}\label{th:S-wp-II} 
Assume that $\bm{f} \in L^2(0, T; \bm L^2(\Omega))$, $g \in H^1(0, T; L^2(\Gamma))$ with $g(0) \in L^2(\Gamma)$ and $\bm{u}_0 \in \bm V_\sigma$. 
Then the solution of {\bf(S-VE)$_{\sigma,\ep}$} satisfies $\bm{u}_\ep \in L^\infty(0, T; \bm V_\sigma) \cap H^1(0, T; \bm H_\sigma)$ with the estimate
\begin{equation}\label{eq:S-wp-II-1}
\begin{aligned}
& \| \bm{u}_\ep' \|_{L^2(0, T; \bm H_\sigma)} + \| \bm{u}_\ep \|_{L^\infty(0, T; \bm V_\sigma)} \le  C\Big(\| \bm{f} \|_{L^2(0, T; \bm L^2)} + \|\bm{u}_0\|_{\bm V_\sigma} + \|g\|_{H^1(0, T; L^2(\Gamma))} \\ 
&\qquad  + \|g(0)\|_{L^2(\Gamma)} + \ep^{\frac{1}{2}} (\|g(0)\|_{L^1(\Gamma)}^{\frac{1}{2}}  + \|g\|_{L^\infty(0, T;L^1(\Gamma))}^{\frac{1}{2}}  + \|g'\|_{L^1(0, T;L^1(\Gamma))}^{\frac{1}{2}}) \Big).
\end{aligned}
\end{equation}
The same regularity and a-priori estimate also hold for $\bm{u}$ of {\bf(S-VI)$_{\sigma}$}. 
\end{theorem}
%%%%

%%%%
\begin{proof}
Same as the proof of Theorem~\ref{th:S-wp-ue}, we employ the Galerkin approximation method. 
And the conclusion of Theorem~\ref{th:S-wp-ue} also holds. In addition, we take the initial value that $\bm{u}_0^m \rightarrow \bm{u}_0$ in $\bm V_\sigma$.  
Let's continue with (Step 2) of the proof for Theorem~\ref{th:S-wp-ue}.

Multiplying \eqref{eq:S-VEem-a} with $c_k'(t)$ and summing up w.r.t. $k$, we obtain  
\begin{equation}\label{eq:prf-S-wp-II-0}
(\bm{u}_\ep^{m \prime}, \bm u_\ep^{m \prime}) + a_0(\bm u_\ep^m, \bm u_\ep^{m \prime}) +\int_{\Gamma}g(t) \alpha_\ep(\bm u_\ep^m)\cdot \bm{u}_\ep^{m \prime}~ds \\ 
 = ( \bm{f}, \bm{u}_\ep^{m \prime})  \le \frac{1}{2}\|\bm{f}\|^2 + \frac{1}{2}\| \bm{u}_\ep^{m \prime}\|^2.
\end{equation}
Integrating the above equation over $[0, T]$ and in view of 
$ a_0(\bm u_\ep^m, \bm u_\ep^{m \prime}) = \frac{1}{2}\frac{d}{dt}a_0(\bm{u}_\ep^m, \bm{u}_\ep^m)$, we get 
\begin{equation}\label{eq:prf-S-wp-II-1}
\begin{aligned}
 &\int_0^t\| \bm u_\ep^{m \prime}(\tau) \|^2~d\tau + \big( a_0(\bm{u}_\ep^m(t), \bm{u}_\ep^m(t)) - a_0(\bm{u}_0^m, \bm{u}_0^m) \big) \\
 &\qquad + 2 \int_0^t \int_{\Gamma} g(\tau) \alpha_\ep(\bm u_\ep^m(\tau))\cdot \bm{u}_\ep^{m \prime}(\tau)~ds d\tau  \le \int_0^t \|\bm{f}(\tau)\|^2~d\tau.
\end{aligned}
\end{equation}
In view of $(\rho_\ep(\bm{u}_\ep^m))' = \alpha_\ep(\bm{u}_\ep^m) \cdot \bm{u}_\ep^{m \prime}$, we apply the integration by parts to calculate the third term of the left-hand side.
\[
\begin{aligned}
& \quad \bigg| \int_0^t \int_{\Gamma} g(\tau) \alpha_\ep(\bm u_\ep^m)\cdot \bm{u}_\ep^{m \prime}~ds d\tau \bigg|= \bigg| \int_0^t \int_{\Gamma} g(\tau)(\rho_\ep(\bm{u}_\ep^m))'~ds d\tau \bigg| \\ 
& = \bigg|- \int_0^t \int_{\Gamma} g'(\tau) \rho_\ep(\bm{u}_\ep^m)~ds d\tau + \int_{\Gamma} g(t) \rho_\ep(\bm{u}_\ep^m(t))~ds  - \int_{\Gamma} g(0) \rho_\ep(\bm{u}_\ep^m(0))~ds \bigg| \\
& \le \bigg| \int_0^t \int_{\Gamma} g' (\rho_\ep(\bm{u}_\ep^m) - |\bm{u}_\ep^m|)~ ds d\tau \bigg| + \bigg|\int_0^t \int_{\Gamma} g' |\bm{u}_\ep^m|~ds d\tau\bigg| + \bigg|\int_{\Gamma} g ( \rho_\ep(\bm{u}_\ep^m) - |\bm{u}_\ep^m|)~ds \bigg|  \\
& \qquad + \bigg|\int_{\Gamma} g |\bm{u}_\ep^m |~ds \bigg| +  \bigg|\int_{\Gamma} g(0) (\rho_\ep(\bm{u}_0^m) - |\bm{u}_0^m|)~ds \bigg| + \bigg| \int_{\Gamma} g(0) |\bm{u}_0^m|~ds\bigg| \\
& \le \ep \| g' \|_{L^1(0, T;L^1(\Gamma))} + \| g' \|_{L^2(0, T; L^2(\Gamma))} \| \bm{u}_\ep^m \|_{L^2(0, T; \bm \Lambda)} + \ep \| g(t) \|_{L^1(\Gamma)} \\
& \qquad  + \| g(t) \|_{L^2(\Gamma)} \| \bm{u}_\ep^m(t) \|_{\bm \Lambda} + \ep \| g(0) \|_{L^1(\Gamma)} + \| g(0) \|_{L^2(\Gamma)} \| \bm{u}_0^m \|_{\bm \Lambda}, 
\end{aligned}
\]
where we have utilized the triangle inequality and \eqref{eq:rho-err}.
Inserting the above estimate into \eqref{eq:prf-S-wp-II-1}, and in view of $H^1(0, T; \bm L^2(\Gamma)) \subset L^\infty(0, T; \bm L^1(\Gamma))$ and $\| \bm{u}_\ep^m \|_{\bm \Lambda} \le C \| \bm{u}_\ep^m \|_{\bm H^1}$, we obtain the following estimate of $\bm{u}_\ep^m$.   
\begin{equation}\label{eq:prf-S-wp-II-2}
  \begin{aligned}
  & \| \bm{u}_\ep^{m \prime} \|_{L^2(0, T; \bm H_\sigma)} + \| \bm{u}_\ep^m \|_{L^\infty(0, T; \bm V_\sigma)}  \le C\Big(\| \bm{f} \|_{L^2(0, T; \bm L^2)} + \|\bm{u}_0\|_{\bm V_\sigma} + \|g\|_{H^1(0, T;L^2(\Gamma))} \\
  & \qquad + \|g(0)\|_{L^2(\Gamma)} + \ep^{\frac{1}{2}} (\|g(0)\|_{L^1(\Gamma)}^{1/2}  + \|g\|_{L^\infty(0, T;L^1(\Gamma))}^{1/2}  + \|g'\|_{L^1(0, T;L^1(\Gamma))}^{1/2}) \Big).
  \end{aligned}
\end{equation}

In view of \eqref{eq:S-est-I} and \eqref{eq:prf-S-wp-II-2}, there is a subsequence $\{ \bm{u}_\ep^m \}_m$ such that, as $m \rightarrow \infty$,  
\[
  \begin{aligned}
    &\bm{u}_{\ep}^m \stackrel{*}{ \rightharpoonup}  \bm{u}_{\ep} \quad &\text{ in } L^\infty(0, T; \bm V_\sigma) ,\\
    &\bm{u}_{\ep}^{m\prime} \rightharpoonup \bm{u}_{\ep}' \quad &\text{ in } L^2(0, T; \bm H_\sigma). 
  \end{aligned}
\]
As demonstrated in Theorem~\ref{th:S-wp-ue} that $\bm{u}_{\ep}$ is the unique solution of {\bf(S-VE)$_{\sigma,\ep}$}. 
The a-priori estimate \eqref{eq:S-wp-II-1} is implied by \eqref{eq:prf-S-wp-II-2}. 

Then, we can extract a subsequence $\{ \bm{u}_{\ep} \}_\ep$ such that, as $\ep \rightarrow 0$, 
\[
  \begin{aligned}
    &\bm{u}_\ep \stackrel{*}{ \rightharpoonup}  \bm{u} \quad & \text{ in } L^\infty(0, T; \bm V_\sigma) ,\\
    &\bm{u}_\ep' \rightharpoonup \bm{u}' \quad &\text{ in } L^2(0, T; \bm H_\sigma). 
  \end{aligned}
\]
As with the proof of Theorem~\ref{th:S-wp-u}, the passage to the limit for $\ep \rightarrow 0$ shows that $\bm{u}$ solves {\bf(S-VI)$_\sigma$}. It is apparent that the estimate \eqref{eq:S-wp-II-1} also holds for $\bm{u}$. 
\end{proof}
%%%%

Assume that $g(0) \in H^1(\Gamma)$, $\bm h \in \bm L^2(\Omega)$ and the initial value $\bm{u}_0 \in \bm V_\sigma$ solve the stationary variational inequality
\begin{equation}\label{eq:VI_u0}
  a_0(\bm{u}_0,\bm{v}-\bm{u}_0)+\int_{\Gamma} g(0)|\bm{v}|ds -\int_{\Gamma}g(0)|\bm{u}_0|ds \geq (\bm h,\bm{v}-\bm{u}_0)\quad \forall \bm{v} \in \bm V_\sigma,
\end{equation}
We denoted by $\bm{u}_{0\ep} \in \bm V_\sigma$ the solution of the regularization problem: 
\begin{equation}\label{eq:VI_u0e}
  a_0(\bm{u}_{0\ep}, \bm{v}-\bm{u}_{0\ep}) +\int_{\Gamma}g(0)\rho_\ep (\bm{v})~ds- \int_{\Gamma} g(0)\rho_\ep (\bm{u}_{0\ep})~ds \geq (\bm h, \bm{v}-\bm{u}_{0\ep}) \quad \forall \bm v \in \bm V_\sigma, 
\end{equation}
or equivalently,
\begin{equation}\label{eq:VE_u0e}
  a_0(\bm{u}_{0\ep}, \bm{v})+\int_{\Gamma} g(0)\alpha_\ep(\bm{u}_{0\ep})\cdot \bm{v}~ds = (\bm h, \bm{v}) \quad \forall \bm v \in \bm V_\sigma.  
\end{equation}
Note that if $\bm u_0 \in \bm H^2(\Omega_a)$, then $\bm h = -2\nu \nabla \cdot \mathbb{D}(\bm u_0)$ in $\Omega_\mathrm{in}$ and $\Omega_\mathrm{out}$ respectively.
%%%%
\begin{lemma}[$(\bm H^2(\Omega_{\mathrm{in}}), \bm H^2(\Omega_{\mathrm{out}}))$ regularity for the stationary problem]\label{la:est-u0}
Given $g(0) \in H^1(\Gamma)$ and $\bm h \in \bm L^2(\Omega)$, the variational inequalities \eqref{eq:VI_u0e} and \eqref{eq:VI_u0} admit unique solutions $\bm{u}_{0\ep} \in \bm V_\sigma$ and $\bm{u}_0 \in \bm V_\sigma$, respectively, satisfying
\begin{subequations}\label{eq:u0/u0e}
\begin{align}
  & \|\bm u_{0\ep}\|_{\bm H^1} \le C\|\bm h\|, \quad \|\bm u_0\|_{\bm H^1} \le C\|\bm h\|, \label{eq:u0/u0e-a} \\ 
  & \| \bm{u}_{0\ep} - \bm{u}_0 \|_{\bm H^1} \rightarrow 0, \quad \text{ as } \ep \to 0. \label{eq:u0/u0e-b}
\end{align}
\end{subequations}

Moreover, we have $\bm{u}_{0\ep}, \bm{u}_0 \in \bm V_\sigma \cap \bm H^2(\Omega_a)$ and
\begin{subequations}\label{eq:u0-u0e-H2}
\begin{align}
& \| \bm{u}_{0\ep} \|_{\bm H^2(\Omega_a)} + \| p_{0\ep} \|_{\bm H^1(\Omega_a)} \le C(\|\bm h\| + \| g(0)\|_{H^1(\Gamma)}), \label{eq:u0e-H2} \\ 
& \| \bm{u}_0 \|_{\bm H^2(\Omega_a)} +\| p_{0} \|_{\bm H^1(\Omega_a)} \le C(\|\bm h\| + \|g(0)\|_{H^1(\Gamma)}). \label{eq:u0-H2}
\end{align}
\end{subequations}

\end{lemma}
%%%%
See Appendix~\ref{App:1} for the proof. 

%%%%
\begin{theorem}\label{th:S-wp-III}
Assume that $\bm{f} \in H^1(0, T; \bm L^2(\Omega))$, $g \in H^1(0, T;L^2(\Gamma))$, $g(0) \in H^1(\Gamma)$. Suppose that $\bm{u}_0 \in \bm V_\sigma \cap \bm H^2(\Omega_a)$ solves \eqref{eq:VI_u0} with some $\bm h  \in \bm L^2(\Omega)$, and $\bm{u}_{0\ep}$ is the solution of the regularization problem \eqref{eq:VE_u0e}. 
We replace the initial condition of {\bf(S-VE)$_{\sigma,\ep}$} with $\bm{u}_\ep(0) = \bm{u}_{0\ep}$. 
In addition to the result of Theorems~\ref{th:S-wp-ue} and \ref{th:S-wp-II}, the solution of {\bf(S-VE)$_{\sigma,\ep}$} satisfies $\bm{u}_\ep' \in L^\infty(0, T; \bm L^2(\Omega)) \cap L^2(0, T; \bm H^1(\Omega))$ with the a-priori estimates: 
\begin{equation}\label{eq:S-wp-III}
\| \bm{u}_\ep'\|_{L^\infty(0, T; \bm L^2)} + \| \bm{u}_\ep'\|_{L^2(0, T; \bm H^1)} \le C \big(\| \bm h\| + \|g'\|_{L^2(0, T;L^2(\Gamma))} + \|\bm{f}\|_{H^1(0, T; \bm L^2)} \big).
\end{equation}
The same estimate is satisfied for $\bm u'$ of {\bf(S-VI)$_{\sigma}$}.

%%%
Moreover, we have $\bm u \in L^\infty(0,T; \bm H^2(\Omega_a))$ and $p \in L^\infty(0,T; H^1(\Omega_a))$ with the estimate 
\begin{equation}\label{eq:u-p-H2}
\|\bm u(t)\|_{\bm H^2(\Omega_a)}+\|p(t)\|_{H^1(\Omega_a)} \le C(\|\bm f(t)\|+\|\bm f'(t)\|+\| g(0)\|_{H^1(\Gamma)}+\|g'(t)\|_{L^2(\Gamma)}+\|\bm h\|).
\end{equation}
\end{theorem}
%%%%

\begin{proof}
We differentiate \eqref{eq:S-VEem-a} w.r.t. $t$ 
\begin{equation}\label{eq:prf-S-wp-III-1}
  (\bm u_\ep^{m \prime\prime}, \bm{w}_k)+\bm{a}_0(\bm u_\ep^{m \prime}, \bm{w}_k) +\int_{\Gamma}g'(t) \alpha_\ep(\bm u_\ep^m)\cdot \bm{w}_k ~ds  + \int_{\Gamma}g(t) \bm{u}_\ep^{m\prime \top}\beta_\ep(\bm{u}_\ep^m)\bm{w}_k ~ds = (\bm{f}', \bm{w}_k).
\end{equation}
Multiplying \eqref{eq:prf-S-wp-III-1} with $c'_k(t)$ and summing up w.r.t. $k$, we get 
\[
\frac{1}{2}\frac{d}{dt} \| \bm u_\ep^{m \prime}(t)\|^2 + \bm{a}_0(\bm u_\ep^{m \prime}, \bm u_\ep^{m \prime}) +\int_{\Gamma}g'(t) \alpha_\ep(\bm u_\ep^m)\cdot \bm u_\ep^{m \prime}~ds + \int_{\Gamma}g(t) \bm{u}_\ep^{m\prime\top} \beta_\ep(\bm{u}_\ep^m) \bm{u}_\ep^{m \prime} ~ds = (\bm{f}', \bm{u}_\ep^{m \prime}).
\]
Integrating the above equation w.r.t. $t$, in view of $|\alpha_\ep(\bm{z})| \le 1$ and \eqref{eq:beta_ep}, we obtain  
\begin{equation}\label{eq:prf-S-wp-III-2}
\sup_{t \in [0, T]}\| \bm u_\ep^{m \prime}(t)\|^2 + \| \bm u_\ep^{m \prime}\|_{L^2(0, T; \bm H^1)}^2 \le C \big( \| \bm u_\ep^{m \prime}(0)\|^2 + \|g'\|_{L^2(0, T;L^2(\Gamma))}^2 + \|\bm{f}'\|_{L^2(0, T; \bm L^2)}^2 \big).
\end{equation}
We have to evaluate $\| \bm u_\ep^{m \prime}(0)\|_{L^2}$ of the right-hand side. 

Note that \eqref{eq:prf-S-wp-II-0} holds at $t=0$, we have
\begin{equation}\label{eq:prf-S-wp-III-3}
(\bm{u}_\ep^{m \prime}(0), \bm{u}_\ep^{m \prime}(0))+\bm{a}_0(\bm{u}_{0\ep}, \bm{u}_\ep^{m \prime}(0)) + \int_{\Gamma}g(0) \alpha_\ep(\bm u_{0\ep})\cdot \bm{u}_\ep^{m \prime}(0)~ds = ( \bm{f}(0), \bm{u}_\ep^{m \prime}(0)).
\end{equation}
Testing \eqref{eq:VE_u0e} by $\bm{v} = \bm{u}_\ep^{m \prime}(0)$, we see that 
\begin{equation}\label{eq:prf-S-wp-III-3'}
\bm{a}_{0}(\bm{u}_{0\ep}, \bm{u}_\ep^{m \prime}(0))+\int_{\Gamma} g(0)\alpha_\ep(\bm{u}_{0\ep})\cdot \bm{u}_\ep^{m \prime}(0)~ds = (\bm h, \bm{u}_\ep^{m \prime}(0)). 
\end{equation}
Together with \eqref{eq:prf-S-wp-III-3}, we have  
\[
\|\bm{u}_\ep^{m \prime}(0)\|^2 = (\bm{f}(0) - \bm h, \bm{u}_\ep^{m \prime}(0)) \le \| \bm{f}(0) - \bm h\| \|\bm{u}_\ep^{m \prime}(0)\|, 
\]
which yields 
\begin{equation}\label{eq:prf-S-wp-III-4}
\|\bm{u}_\ep^{m \prime}(0)\| \le \| \bm{f}(0) - \bm h\| \le \| \bm{f}(0) \| + \| \bm h\|. 
\end{equation}
Inserting \eqref{eq:prf-S-wp-III-4} into \eqref{eq:prf-S-wp-III-2} results in
\[
\sup_{t\in [0, T]}\| \bm u_\ep^{m \prime}(t)\|^2 + \| \bm u_\ep^{m \prime}\|_{L^2(0, T; \bm H^1)}^2 \le C \big(\|\bm h\|^2 + \|g'\|_{L^2(0, T;L^2(\Gamma))}^2 + \|\bm{f}\|_{H^1(0, T; \bm L^2)}^2 \big).
\]
Hence, there exists a subsequence $\{ \bm{u}_\ep^m \}_m$ such that as $m\rightarrow \infty$, 
\[
  \begin{aligned}
    &\bm{u}_{\ep}^{m \prime} \stackrel{*}{ \rightharpoonup}  \bm{u}_{\ep}' \quad \text{ in } L^\infty(0, T; \bm L^2) ,\\
    &\bm{u}_{\ep}^{m\prime} \rightharpoonup \bm{u}_{\ep}' \quad \text{ in } L^2(0, T; \bm V_\sigma),   
  \end{aligned}
\]
and the a-priori estimate \eqref{eq:S-wp-III} is satisfied. 
Using the results of Theorem~\ref{th:S-wp-ue}, we assert that $\bm{u}_{\ep}$ uniquely solves  {\bf(S-VE)$_{\sigma,\ep}$}. 

Furthermore, we can find a subsequence $\{ \bm{u}_\ep \}_\ep$ such that as $\ep \rightarrow 0$, 
\[
\begin{aligned}
  &\bm{u}_\ep' \stackrel{*}{ \rightharpoonup}  \bm{u}' \quad \text{in } L^\infty(0, T; \bm L^2) ,\\
  &\bm{u}_{\ep}' \rightharpoonup \bm{u}' \quad \text{in } L^2(0, T; \bm V_\sigma),   
\end{aligned}
\]
and the estimate \eqref{eq:S-wp-III} holds for $\bm u$. Together with the results of Theorem~\ref{th:S-wp-u} and \eqref{eq:u0/u0e-b}, we conclude that $\bm{u}$ is the unique solution of {\bf(S-VI)$_\sigma$}. 

%%%
Since $\bm f-\bm u_\ep' \in  L^\infty(0,T; \bm L^2(\Omega))$, we can apply Lemma \ref{la:est-u0} to the folllowing problem. For a.e. $t \in [0, T]$,
\begin{subequations}\label{eq:S-Pt}
  \begin{align}
    -\nabla \cdot \bm{\sigma}(\bm u(t), p(t)) &= \bm f(t)-\bm u'(t) \quad &&\text { in } \Omega, \label{eq:S-Pt-a}\\
    \nabla \cdot \bm u(t) &=0 \quad &&\text { in } \Omega , \label{eq:S-Pt-b}\\
    \bm u(t)&=0 \quad &&\text { on } \Gamma_{D}, \label{eq:S-Pt-c} \\
    [\bm u(t)]=0, ~[\bm \sigma(\bm u(t), p(t))\bm n] & \in g(t) \partial |\bm u(t)| \quad&&\text { on } \Gamma. \label{eq:S-Pt-d}
  \end{align} 
\end{subequations}
Hence, $(\bm u, p) \in L^\infty(0,T; \bm H^2(\Omega_a)) \times L^\infty (0,T; H^1(\Omega_a))$ and \eqref{eq:u0-p0-H2} implies \eqref{eq:u-p-H2}.
\end{proof}
%%%

%%
\section{The well-posedness of the Navier-Stokes case} \label{sec:4}
In this section, we consider the class of Leray-Hopf type weak solution ~\cite{Leray34, Hopf51}
\[
  \bm u \in L^2(0, T; \bm V_\sigma) \cap L^\infty (0, T; \bm H_\sigma).
\]
The pressure $p$ is usually absent in the weak solution. 
We also consider the Ladyzhenskaya type strong solution \cite{Ladyzhenskaya69}
\[
  \begin{cases}
    \bm u \in L^\infty(0, T; \bm V_\sigma),~ \bm u' \in L^\infty(0, T; \bm H_\sigma) \cap L^2(0,T;\bm V_\sigma); \\ 
    p \in L^\infty (0,T; Q),
  \end{cases}
\]
where the local uniqueness and existence are guaranteed in 3D case. 

Throughout this section, we still assume that
\begin{equation}\label{eq:ass-NS-w}
  \text{For a.e. } t,~\bm f(t) \in \bm V', ~g(t) \in L^2(\Gamma) \text{ with } g > 0 \text{ and } \bm u_0 \in \bm H_\sigma.
\end{equation}
Further regularity assumptions of the data will be provided later.

%%%%%%%%
\subsection{Weak formulations}
The derivation of the weak formulation is almost the same as in Section~\ref{sec:3} with an extra nonlinear term $a_1(\bm u, \bm u, \bm v)$. 
Therefore, we give the variational inequality directly.

%%%
{\bf(NS-PDE)}  For a.e $t \in (0,T)$, find $(\bm u(t),p(t)) \in \bm V_\sigma \times \mr{Q}$ such that $\bm u'(t) \in \bm V_\sigma'$, $\bm{u}(0) = \bm{u}_0$, 
\begin{subequations}\label{eq:NS-PDE}
\begin{align}
  & \langle \bm{u}',\bm v\rangle + a_1(\bm u, \bm u, \bm v)+a_0(\bm u,\bm v) + b(\bm v, p) = \langle \bm{f},\bm v\rangle  \quad \forall \bm v \in \bm V^0, \\ 
  & |[\bm\sigma \bm{n}](t)| \le g(t),~ [\bm\sigma \bm{n}](t) \cdot \bm{u}- g(t)|\bm{u}| = 0 \text{ a.e. on } \Gamma,
\end{align}
\end{subequations}
where $[\bm\sigma \bm{n}] \in \bm \Lambda_{1/2}$ is defined by
\begin{equation}\label{eq:def-sign}
\langle [\bm\sigma \bm{n}], \bm v \rangle_{\bm \Lambda_{1/2}} := \langle \bm{f},\bm v\rangle - \langle \bm{u}',\bm v\rangle - a_1(\bm u, \bm u, \bm v) - a_0(\bm u,\bm v) - b(\bm v, p) \quad \forall \bm v \in \bm V.
\end{equation}
%

%%%
We can easily find that if $\bm u$ is a classic solution of {\bf(NS-P)}, then it also satisfies {\bf(NS-PDE)}. 
A sufficiently smooth weak solution is a classical solution. 

%%%Problem NS-VI
{\bf(NS-VI)$_\sigma$}  For a.e $t \in (0,T)$, find $\bm u(t) \in \bm V_\sigma$ such that $\bm{u}'(t) \in \bm V_\sigma'$, $\bm u(0) = \bm u_0$ and 
\begin{equation}\label{eq:NS-VI-sig}
  \langle \bm u',\bm v-\bm u \rangle +a_1(\bm u,\bm u, \bm v -\bm u)+a_0(\bm u,\bm v-\bm u) +j(t;\bm v)-j(t;\bm u) \geq \langle \bm f,\bm v-\bm u\rangle  \quad \forall \bm v  \in V_\sigma.
\end{equation}
%

%%%% equivalence thm PDE-VI
\begin{proposition}\label{prop:equi-NS}
The solution of {\bf(NS-PDE)} solves {\bf(NS-VI)$_\sigma$}. 
Conversely, if {\bf(NS-VI)$_\sigma$} admits a solution $\bm u(t)$, then there exists at least one $p(t) \in \mr{Q}$ such that $(\bm u(t),p(t))$ solves  {\bf(NS-PDE)}.
If another $p(t)^* \in \mr{Q}$ satisfies the same condition, then there exists a unique $\delta(t) \in \mathbb R$ such that $[p(t)]=[p^*(t)] +\delta(t)$. 
Furthermore, if $\bm{u}\cdot \bm{n}(t) \neq 0$ on $\Gamma$, then $\delta(t) = 0$, i.e., the associated pressure $p(t)$ is uniquely determined.
\end{proposition}
%%%%
The proof of Proposition ~\ref{prop:equi-NS} is similar to Theorem ~\ref{th:equi-S}.

We still approximate the nonlinear non-differentiable functional $j(t;\cdot)$ by the regularization functional $j_\ep(t;\cdot)$.

%%% Problem NS-VI-ep
{\bf(NS-VI)$_{\sigma,\ep}$}  For a.e. $t\in (0,T)$, find $\bm u_\ep(t) \in \bm V_\sigma$ such that $\bm u_\ep'(t) \in \bm V_\sigma'$, $\bm u_\ep(0) = \bm u_0$ and
\begin{equation}\label{eq:NS-VIe}
  \langle \bm u'_\ep, \bm{v}-\bm u_\ep\rangle + a_1(\bm u_\ep, \bm u_\ep,\bm{v}-\bm u_\ep)+a_0(\bm u_\ep,\bm{v}-\bm u_\ep) +j_\ep(t;\bm{v})-j_\ep(t;\bm u_\ep) \geq \langle \bm f,\bm{v}-\bm u_\ep\rangle \quad \forall \bm{v}  \in \bm V_\sigma.
\end{equation}
Since $\rho_\ep$ is differentiable, the above variational inequality is equivalent to the following nonlinear variational problem. 

%%% Problem NS-VE-ep
{\bf(NS-VE)$_{\sigma,\ep}$}  For a.e. $t\in (0,T)$, find $\bm u_\ep(t) \in \bm V_\sigma$ such that $\bm u'_\ep(t) \in \bm V_\sigma'$, $\bm u_\ep(0) = \bm u_0$ and   
\begin{equation}\label{eq:NS-VEe}
  \langle \bm u'_\ep,\bm{v} \rangle+ a_1(\bm u_\ep,\bm u_\ep,\bm{v})+a_0(\bm u_\ep,\bm{v})+\int_{\Gamma}g(t) \alpha_\ep(\bm u_\ep)\cdot \bm{v}~ds =\langle\bm f,\bm{v}\rangle \quad \forall \bm{v}  \in \bm V_\sigma.
\end{equation}
%

% {\bf(NS-VI)$_{\sigma,\ep}$} and {\bf(NS-VE)$_{\sigma,\ep}$} are equivalent.

%%%%%%%%
\subsection{The existence theorem of weak solution}
%添加定义 H^\gamma 和Fourier变换
Before proving the existence theorem of the weak solution of {\bf(NS-VE)$_{\sigma,\ep}$}, we introduce the notation of fractional derivative in time of a function.

Assume that $X_0, X_1$ are Hilbert spaces with $X_0 \subset X_1$, and $\bm v$ is a function from $\mathbb R$ into $X_1$. We denote by $\hat{\bm v}$ its Fourier transform 
\begin{equation}\label{eq:def_fourier}
  \hat{\bm v}(\tau) = \int_{- \infty}^{+\infty } e^{-2i\pi t\tau} \bm v (t)~dt. 
\end{equation}
and the derivative in $t$ of order $\gamma$ of $\bm v$ is 
$
\widehat{D_t^{\gamma}\bm v(\tau)} = (2\pi \tau)^{\gamma} \hat{\bm v}(\tau).
$
For a given $\gamma >0$, define the space
\[
  \mathcal{H}^\gamma (\mathbb R; X_0,X_1) :=\{\bm v \in  L^2(\mathbb R; X_0), D_t^\gamma \bm v\in L^2(\mathbb R; X_1)\}
\]
with the norm 
\[
\| \bm v\| _{ \mathcal{H}^\gamma (\mathbb R; X_0,X_1)} =\{  \|\bm v\|_{L^2(\mathbb R; X_0)} + \| |\tau|^{\gamma} \hat{\bm v}\|_{D_t^\gamma \bm v\in L^2(\mathbb R; X_1)}\}^{1/2}
\]

% 定义算子 A和 B
For $\bm u, \bm v \in \bm V_\sigma$, we define
\[
  \begin{aligned}
    & \mathcal A: \bm V_\sigma \to \bm V_\sigma' , ~ \langle \mathcal A \bm u, \bm v\rangle :=a_0(\bm u, \bm v),  \\ 
    & \mathcal B:   \bm V_\sigma \to \bm V_\sigma', ~ \langle \mathcal B \bm u, \bm v\rangle := a_1(\bm u, \bm u, \bm v),\\ 
    &\mathcal C:  \bm V_\sigma \to \bm V_\sigma', ~ \langle \mathcal C \bm u, \bm v\rangle := \int_\Gamma g(t)\alpha_\ep(\bm u)\cdot \bm v~ds.
  \end{aligned}
\]
It is easy to see that 
\begin{equation}\label{eq:est-opt}
  \|\mathcal A \bm u\|_{\bm V_\sigma'}   \le C\|\bm u\|_{\bm H^1},  ~\|\mathcal B \bm u\|_{\bm V_\sigma'}   \le C\|\bm u\|_{\bm H^1}^2, ~\|\mathcal C \bm u\|_{\bm V_\sigma'}   \le C\|g(t)\|_{L^2(\Gamma)} ~(\text{by } |\alpha_\ep(\cdot) |\le 1).
\end{equation}
%

%%% thm VE_ep
\begin{theorem}\label{th:NS-wp-ue}
Under the assumptions of \eqref{eq:ass-NS-w}, there exists at least a weak solution $\bm u_\ep \in L^2(0, T; \bm V_\sigma )\cap L^\infty(0, T; \bm H_\sigma)$ and $\bm u_\ep' \in L^r(0, T; \bm V_\sigma')(r=2 \text{ when }d=2, r=4/3 \text{ when }d=3)$ of {\bf(NS-VE)$_{\sigma,\ep}$}.
\end{theorem}
%%%

\begin{proof}
The proof can be divided into five steps.  The priori estimate is similar to Theorem \ref{th:S-wp-ue},   and additionally, we prove $  \tilde{\bm u}^m_\ep \text{ belongs to a bounded set of } \mathcal{H}^\gamma(\mathbb{R} ;\bm V_\sigma, \bm H_\sigma)$ in (Step 3).

%%%
(Step 1) {\bf Galerkin's method}. For fixed $m$, find $c_k(t)$ such that $\bm{u}_\ep^m=\sum_{k=1}^m c_k(t) \bm w_k$ satisfies
\begin{subequations}\label{eq:w-VEem}
  \begin{align}
    \langle\bm{u}_{\ep}^{m\prime},\bm w_k\rangle+a_1(\bm{u}_\ep^m,\bm{u}_\ep^m,\bm w_k)+a_0(\bm{u}_\ep^m ,\bm w_k)+\int_\Gamma g(t)\alpha_\ep(\bm{u}_\ep^m)\cdot \bm w_k ~ds &=\langle \bm f,\bm w_k\rangle, \label{eq:w-VEem-a} \\
    \bm{u}_\ep^m (0)&=\bm{u}_0^m.\label{eq:w-VEem-b}
  \end{align}
\end{subequations}
By the standard existence theory of finite dimension ODEs system, $\exists~T_1 > 0$ such that \eqref{eq:w-VEem} admits unique local solutions $c_k(t)\in C^2([0, T_1])( k=1,\ldots,m)$. 
Next, we can extend $T_1$ to $T$ by a-priori estimate.

%%%
(Step 2) {\bf A-priori estimate}. Multiplying \eqref{eq:w-VEem-a} by $c_k(t)$ and summing for $k=1,\ldots m$, together with ~\eqref{eq:a1=0}, we obtain
\[ 
  \langle\bm{u}_\ep^{m\prime},\bm{u}_\ep^m\rangle+a_0(\bm{u}_\ep^m,\bm{u}_\ep^m) +\int_\Gamma g(t)\alpha_\ep(\bm{u}_\ep^m)\cdot\bm{u}_\ep^m ~ds=\langle \bm f,\bm{u}_\ep^m\rangle, 
\]
where $a_1(\bm{u}_\ep^m,\bm{u}_\ep^m,\bm{u}_\ep^m)$ vanish because we consider the frictional interface problem \eqref{eq:P-e} instead of frictional boundary problem (see Remark \ref{rk:FIC}).

Similar to \eqref{eq:S-est-0}, we get
\begin{equation}\label{eq:NS-est-0}
  \frac{d}{d t}\|\bm{u}_{\ep}^m(t)\|^2+c\|\bm{u}_{\ep}^m\|_{\bm{H}^1}^2 \leq \frac{1}{c}\|\bm{f}\|^2.
\end{equation}
Then, integrating the above equation w.r.t. $t$, we obtain
\begin{equation}\label{eq:NS-est-1}
  \sup_{t \in [0,T]}\|\bm{u}_\ep^m(t)\|_{\bm H_\sigma}^2 + c \|\bm{u}_\ep^m \|_{L^2(0,T;\bm V_\sigma)}^2 
  \leq C \| \bm f \|_{L^2(0,T;\bm V')}^2 + \|\bm{u}_0\|^2   
\end{equation}
which implies that $L^\infty(0,T;\bm H_\sigma)$ and $L^2(0,T;\bm V_\sigma)$ norms of $\bm{u}_\ep^m$ are bounded independent of $m$ and $\ep$.

(Step 3) To pass to the limit in the nonlinear term $a_1$, we need an estimate of the fractional derivative in time of $\bm u_\ep^m$. Setting 
$
  \bm f^m =   \bm f-\nu \mathcal A\bm u_\ep^m - \mathcal B \bm u_\ep^m -\mathcal C  \bm u_\ep^m , 
$
because of \eqref{eq:est-opt}, we have 
\begin{equation}\label{eq:est-fm}
 \|\bm f^m(t)\|_{\bm V_\sigma'} \le C\left(\|\bm f(t)\|_{\bm V_\sigma'} +\nu \|\bm u_\ep^m(t)\|_{\bm V_\sigma} + \|\bm u_\ep^m(t)\|_{\bm V_\sigma}^2 + \|g(t)\|_{L^2(\Gamma)}\right).
\end{equation}

Now we extend ${\bm u}_\ep^m, \bm f^m$  by $0$ outside $[0,T]$  in ~\eqref{eq:w-VEem-a}  denoted by $\tilde{\bm u}_\ep^m,  \tilde{\bm f}^m $ to get  
\begin{equation}\label{eq:extend}
  \frac{d}{dt}(\tilde{\bm u}_\ep^m, \bm w_k) = \langle \tilde{\bm f}^m, \bm w_k \rangle + (\bm u_0^m, \bm w_k) \delta_0 - (\bm u_\ep^m(T), \bm w_k) \delta_T, \quad k=1, \ldots, m
\end{equation}
where $\delta_0, \delta_T$ are Dirac distributions at $0$ and $T$.  
After taking  the Fourier transform, \eqref{eq:extend} results in
\begin{equation}\label{eq:F-tran}
  2i \pi\tau (\hat{\bm u}_\ep^m, \bm w_k) = \langle \hat{\bm f}^m, \bm w_k \rangle + (\bm u_0^m, \bm w_k)  - (\bm u_\ep^m(T), \bm w_k)e^{-2i\pi \tau T}, 
\end{equation}
where $\hat{\bm u}_\ep^m$ and $\hat{\bm f}^m$ are the Fourier transform of ${\bm u}_\ep^m$ and $\bm{f}^m$, respectively. 
We multiply \eqref{eq:F-tran} by $\hat{c}_k(\tau)$ (the Fourier transform of $\tilde c_k(t)$) and add the resulting equations for $k = 1, \ldots , m$, obtaining
\begin{equation}\label{eq:F-tran-1}
  2i \pi\tau \|\hat{\bm u}_\ep^m(\tau)\|^2 = \langle \hat{\bm f}^m(\tau),\hat{\bm u}_\ep^m(\tau) \rangle + (\bm u_0^m, \hat{\bm u}_\ep^m(\tau) )  - (\bm u_\ep^m(T), \hat{\bm u}_\ep^m(\tau) )e^{-2i\pi \tau T}.
\end{equation}
Via a similar arguement to~\cite[Proof Theorem III.3.1]{Temam77}, one can show that

\begin{equation}\label{eq:u_Fourier}
  \int_{-\infty}^{+\infty} |\tau|^{2\gamma} \|\hat{\bm u}^m_\ep(\tau)\|^2~dt \le \text{const}, \quad \text{for some } 0< \gamma < \frac{1}{4}.
\end{equation}
Together with $\bm u_\ep^m \in L^2(0, T; \bm V_\sigma)$, we conclude that 
\begin{equation}\label{eq:H-gamma}
  \tilde{\bm u}^m_\ep \text{ belongs to a bounded set of } \mathcal{H}^\gamma(\mathbb{R} ;\bm V_\sigma, \bm H_\sigma).
\end{equation}
%

%%%
(Step 4) {\bf Passing to the limit for $m \to \infty$}. We can construct a subsequence of $\{\bm{u}_\ep^m\}_{m=1}^\infty$, denoted by ${\bm{u}_\ep^m}$ again such that, as $m \to \infty$,
\[
  \begin{aligned}
    &\bm{u}_{\ep}^m \rightharpoonup  \bm{u}_{\ep} \quad\text{ in } L^2(0,T;\bm V_\sigma) ,\\
    &\bm{u}_{\ep}^m \stackrel{*}{\rightharpoonup}  \bm{u}_{\ep} \quad\text{ in } L^\infty(0, T; \bm H_{\sigma}).
  \end{aligned}
\]
%
% 引用紧定理 Theorem III.2.2 得到强收敛
Due to \eqref{eq:H-gamma} and ~\cite[Theorem III.2.2]{Temam77}, we have
\begin{equation}\label{eq:L2-strong}
  \bm{u}_{\ep}^m \rightarrow \bm{u}_{\ep} \quad\text{ in } L^2(0, T; \bm H_{\sigma}).
\end{equation}

Multiplying \eqref{eq:w-VEem-a} by $\psi(t)\in C([0,T])$ with $\psi(T)=0$, using the integration by parts, we get
\[
\begin{aligned}
  &-\int_0^T(\bm{u}_\ep^m,\psi'(t)\bm w_k)~dt - (\bm{u}_0^m,\psi(0)\bm w_k)+\int_0^T a_0(\bm{u}_\ep^m ,\psi(t)\bm w_k)~dt  \\ 
  &+\int_0^T a_1(\bm{u}_\ep^m,\bm{u}_\ep^m,\psi(t)\bm w_k)~dt  +\int_0^T\int_\Gamma g\alpha_\ep(\bm{u}_\ep^m)\cdot\psi(t)\bm w_k ~ds dt =\int_0^T\langle f,\psi(t)\bm w_k \rangle ~dt.
\end{aligned}
\]
%

% as in th3.2...
The convergence results of the linear terms can be obtained easily with $m \to \infty$. The nonlinear terms are calculated as follows. As $m \to \infty$,
\begin{equation}\label{eq:cong-a1}
\begin{aligned}
  &\bigg|\int_0^T a_1(\bm{u}_\ep^m,\bm{u}_\ep^m ,\psi(t)\bm w_k)~dt -\int_0^T a_1(\bm{u}_\ep ,\bm{u}_\ep,\psi(t)\bm w_k)~dt \bigg| \\
  =&\bigg|\int_0^T a_1(\bm{u}_\ep^m,\psi(t)\bm w_k,\bm{u}_\ep-\bm{u}_\ep^m)~dt +\int_0^T a_1(\bm{u}_\ep-\bm{u}_\ep^m,\psi(t)\bm w_k,\bm{u}_\ep)~dt \bigg| \\
  \leq& C \max |\psi(t)| \|\nabla \bm w_k\|_{\bm L^\infty}\|\bm{u}_\ep^m-\bm{u}_\ep\|_{L^2(0, T; \bm H_\sigma)}(\|\bm{u}_\ep^m\|_{L^2(0, T; \bm H_\sigma)}+\|\bm{u}_\ep\|_{L^2(0, T; \bm H_\sigma)}) \to 0,
\end{aligned}
\end{equation}
where we have applied the \eqref{eq:sys} and \eqref{eq:L2-strong}. By \eqref{eq:alpha} and the Lebesgue convergence theorem, we get
\[
  \int_0^T\int_\Gamma g\alpha_\ep(\bm{u}_\ep^m)\cdot\psi(t)\bm w_k ~ds dt \rightarrow \int_0^T\int_\Gamma g\alpha_\ep(\bm{u}_\ep)\cdot\psi(t)\bm w_k ~ds dt. \quad (m \to \infty)
\]

The remaining verification of the initial condition $\bm u_\ep(0)=\bm u_0$ are the same as Theorem ~\ref{th:S-wp-ue}.
Hence, $\bm{u}_\ep$ is a weak solution of {\bf(NS-VE)$_{\sigma,\ep}$}.

%%%%
(Step 5) We define $\bm u_\ep'$ in the weak sense: for any $t\in (0,T]$ and $\psi(t) \in C_0^\infty((0,T))$,
\begin{equation}\label{eq:NS-VEe-w}
  \int_0^T \left(\langle \bm{u}_\ep',\bm v\rangle  + a_0(\bm{u}_\ep, \bm v)
  + a_1(\bm{u}_\ep,\bm{u}_\ep,\bm v) +\int_\Gamma g \alpha_\ep(\bm{u}_\ep)\cdot\bm v ~ds - \langle f,\bm v \rangle\right)\psi(t) dt =0.
\end{equation}
When $d = 2$, by using \eqref{eq:a1=0} and \eqref{eq:a1-2d}, we get
\[
  \begin{aligned}
  &\big| \int_0^T a_1(\bm{u}_{\ep}, \bm{u}_{\ep}, \bm{v}) ~dt \big| =\big|\int_0^T a_1(\bm{u}_{\ep},\bm{v},\bm{u}_{\ep}) ~dt \big| \le \int_0^T \| \bm{u}_\ep\|_{\bm L^2} \| \bm{v}\|_{\bm H^1} \| \bm{u}_\ep\|_{\bm H^1} ~dt \\
  \le &  \|\bm{u}_\ep\|_{L^\infty(0,T;\bm L^2)} \|\bm{v}\|_{L^2(0,T;\bm V_\sigma)}  \|\bm{u}_\ep\|_{L^2(0,T;\bm V_\sigma)}. 
  \end{aligned}
\]
When $d=3$, we have
\[
  \begin{aligned}
  &\big| \int_0^T a_1(\bm{u}_{\ep}, \bm{u}_{\ep}, \bm{v}) ~dt \big| =\big|\int_0^T a_1(\bm{u}_{\ep},\bm{v},\bm{u}_{\ep}) ~dt \big| \le \int_0^T \| \bm{u}_\ep\|^{\frac{1}{2}}_{\bm L^2} \| \bm{v}\|_{\bm H^1} \| \bm{u}_\ep\|^{\frac{3}{2}}_{\bm H^1} ~dt \\
  \le &  \|\bm{u}_\ep\|^{\frac{1}{2}}_{L^2 (0,T;\bm L^2)} \|\bm{v}\|_{L^4(0,T;\bm V_\sigma)}  \|\bm{u}_\ep\|^{\frac{3}{2}}_{L^2(0,T;\bm V_\sigma)}. 
  \end{aligned}
\]
It is easy to obtain $\bm{u}'_{\ep} \in L^r(0,T;\bm V_\sigma')$ form \eqref{eq:NS-VEe-w} by using H\"older's inequality.
\end{proof}
%%%

Next, by passing to the limit $\ep \to 0$, we can obtain the existence of {\bf(NS-VI)$_\sigma$}.

%%% thm V
\begin{theorem}\label{th:NS-wp-u}
Under the assumptions of \eqref{eq:ass-NS-w}, there exists at least a weak solution 
\[
\bm u \in L^2(0, T; \bm V_\sigma )\cap L^\infty(0, T; \bm H_\sigma) \text{ with } \bm u' \in L^r(0, T; \bm V_\sigma')
\]
of {\bf(NS-VI)$_{\sigma}$}. Moreover, when $d=2$, the weak solution is unique.
\end{theorem}
%%%

\begin{proof}
As a result of Theorem ~\ref{th:NS-wp-ue}, there exists a subsequence of $\{ \bm{u}_\ep \}_\ep$ such that 
  \[
  \begin{aligned}
  &\bm{u}_{\ep} \rightharpoonup  \bm{u}\quad\text{ in } L^2(0, T; \bm V_\sigma) , \\
  &\bm{u}_{\ep} \stackrel{*}{\rightharpoonup}  \bm{u}\quad\text{ in } L^\infty(0, T; \bm H_{\sigma}) , \\
  &\bm{u}'_{\ep} \rightharpoonup  \bm{u}'\quad\text{ in } L^r(0, T; \bm V_\sigma').
  \end{aligned}
\]
%
% \ep \to 0的强收敛
By the compactness theorem ~\cite[Theorem III.2.3]{Temam77}, we have
\[
  \bm{u}_{\ep} \to \bm{u} \quad\text{ in } L^2(0, T; \bm H_{\sigma}).
\]
Integrating \eqref{eq:NS-VIe} over $[0,T]$, we have
\begin{equation}\label{eq:w-est-5}
  \int_0^T \Big(\langle\bm{u}_\ep',\bm{v}-\bm{u}_\ep\rangle + a_1(\bm{u}_\ep,\bm{u}_\ep ,\bm{v}-\bm{u}_\ep) + a_0(\bm{u}_\ep ,\bm{v}-\bm{u}_\ep) + j_\ep (\bm{v})-j_\ep (\bm{u}_\ep ) \Big)~dt \geq \int_0^T \langle \bm{f},\bm{v}-\bm{u}_\ep\rangle~dt. 
\end{equation}
Here, we only estimate $a_1$, and the estimates of other terms are the same as in Theorem \ref{th:S-wp-u}.
\[
\begin{aligned}
  & \lim_{\ep \to 0}\int_0^T a_1(\bm{u}_\ep,\bm{u}_\ep ,\bm{v}-\bm{u}_\ep)~dt = \lim_{\ep \to 0} \int_0^T a_1(\bm{u}_\ep, \bm{u}_\ep, \bm{v})~dt  \\ 
  & \quad = \int_0^T a_1(\bm{u}, \bm{u}, \bm{v})~dt = \int_0^T a_1(\bm{u}, \bm{u}, \bm{v} -\bm{u})~dt \quad  (\text{ by \eqref{eq:a1=0} and \eqref{eq:cong-a1}})
\end{aligned}
\]  
Taking the lower limit for \eqref{eq:w-est-5} and applying Lebesgue's convergence theorem, we conclude  \eqref{eq:NS-VI-sig}. 

%%%
Then, we prove the uniqueness of weak solution in 2D. 
Suppose {\bf(NS-VI)$_\sigma$} admits two solutions $\bm{u}_1$ and $\bm{u}_2$. Set $\bm{w} = \bm{u}_1-\bm{u}_2$. 
Since $\bm u' \in L^2(0, T; \bm V_\sigma')$ holds in 2D, we can test \eqref{eq:NS-VI-sig} of $\bm u_1$ (resp. $\bm u_2$) by $\bm u_2$ (resp. $\bm u_1$),
Then adding the two inequalities together, we obtain 
\begin{equation}\label{eq:2d-uni-1}
\int_0^T \int_\Omega \bm{w}' \cdot \bm{w}~dxdt + \int_0^T a_0(\bm{w}, \bm{w})~dt + \int_0^T a_1(\bm w,\bm u_1, \bm w) ~dt \le 0. 
\end{equation}
Together with \eqref{eq:a1-2d} and Young's inequality, we get
\[
  \begin{aligned}
    \frac{1}{2}\frac{d}{dt} \|\bm w\|^2 +c \|\bm w\|_{\bm H^1}^2 &\le C \|\bm w\|_{ \bm L^2} \|\bm w\|_{\bm H^1}\|\bm{u}_1\|_{\bm H^1} \\ 
    & \le c \|\bm w\|_{\bm H^1}^2+ C\|\bm w\|_{\bm L^2}^2\|\bm{u}_1\|_{\bm H^1}^2.
  \end{aligned}
\]
By Gronwall's lemma, we obtain
\[
  \|\bm w(t)\|^2 \le e^{\int_0^T C\|\bm{u}_1\|_{\bm H^1}^2~dt}  \|\bm w(0)\|^2.
\]
Since $\bm{u}_1 \in L^2 (0, T; \bm V_\sigma)$ and $\bm w(0) =0$, we conclude $\bm{u}_1(t) = \bm{u}_2(t)$ for a.e. $t\in [0,T]$.
\end{proof}
%%%

%%%
\begin{remark}\label{rk:no-unique}
  In 3D case, we only have $\bm u' \in L^{\frac{4}{3}}(0, T; \bm V_\sigma')$, which means $\bm u$ cannot be a test function for \eqref{eq:NS-VI-sig}. 
Thus, it is difficult to guarantee the uniqueness of weak solution in 3D.
\end{remark}
%%%

%%%%%%%%
\subsection{The existence theorem of strong solution}
In this section, we consider the Ladyzhenskaya type strong solution.
We assume that 
\begin{equation}\label{eq:ass-NS}
\begin{aligned}
& \bm f \in H^1(0, T; \bm L^2(\Omega)), ~g \in  H^1(0,T; L^2(\Gamma) ) \text{ with } g(0) \in H^1 (\Gamma), \\ 
& \bm{u}_0 \in  \cap \bm V_\sigma \cap \bm H^2(\Omega_a) \text{ sloves } \eqref{eq:VI_u0} \text{ with some } \bm h \in \bm L^2(\Omega). 
\end{aligned}
\end{equation}
%

%%%
\begin{theorem}\label{th:NS-strong}
Under the assumptions \eqref{eq:ass-NS}, 
we replace the initial condition of {\bf(NS-VE)$_{\sigma,\ep}$} with $\bm u_\ep(0) = \bm u_{0\ep}$ where $\bm u_{0\ep}$ is the solution of \eqref{eq:VE_u0e}. 
In addition to the result of Theorem \ref{th:NS-wp-ue}, the weak solution $\bm u_\ep$ of {\bf(NS-VE)$_{\sigma,\ep}$} also satisfies 
\begin{equation}\label{eq:NS-strong}
\bm{u}_\ep \in L^\infty(0, T; \bm V_\sigma), ~\bm{u}_\ep' \in L^\infty(0, T; \bm H_\sigma) \cap L^2(0, T; \bm V_\sigma), 
\end{equation}
when $d=2$. ~\eqref{eq:NS-strong} holds for a smaller time interval $[0, T_1]$ when $d=3$.

The same regularity holds for the solution $\bm u$ of {\bf(NS-VI)$_{\sigma}$} and it is unique.

\end{theorem}

%%%
\begin{proof}
In the same way as the proof of Theorem~\ref{th:NS-wp-ue}, we utilize the Galerkin approximation method, and the conclusion remains valid. We proceed with (Step 2) of the proof.

Differentiate \eqref{eq:w-VEem-a} w.r.t. $t$  and multiply it by $c_k'(t) $ to get
\begin{equation}\label{eq:dif-t-1}
  \begin{aligned}
    &\frac{1}{2}\frac{d}{dt}\|\bm{u}_\ep^{m\prime}(t)\|^2 +a_0(\bm{u}_\ep^{m\prime} ,\bm{u}_\ep^{m\prime})+a_1(\bm{u}_\ep^{m\prime},\bm{u}_\ep^m ,\bm{u}_\ep^{m\prime}) + \int_\Gamma g'\alpha_\ep(\bm{u}_\ep^m)\cdot \bm{u}_\ep^{m\prime} ~ds  \\
    &\qquad +\int_\Gamma g(\bm{u}_\ep^{m\prime})^\top \beta(\bm{u}_\ep^m)  \bm{u}_\ep^{m\prime} ~ds= (\bm f',\bm{u}_\ep^{m\prime}),
  \end{aligned}
\end{equation}
where $a_1(\bm{u}_\ep^{m},\bm{u}_\ep^{m\prime} ,\bm{u}_\ep^{m\prime})$ vanish by using \eqref{eq:a1=0}.

Each term of \eqref{eq:dif-t-1} is calculated as follows:
\begin{subequations}\label{eq:trem3-6}
  \begin{align}
    \begin{split} \label{eq:term3}
      a_1(\bm{u}_\ep^{m\prime},\bm{u}_\ep^m ,\bm{u}_\ep^{m\prime}) &\leq C\|\bm{u}_\ep^{m\prime}\|_{\bm L^2}\|\bm{u}_\ep^{m\prime}\|_{\bm H^1}\|\bm{u}_\ep^{m}\|_{\bm H^1} \quad (\text{by }\eqref{eq:a1-2d}) \\
      &\leq \frac{c}{6} \|\bm{u}_\ep^{m\prime}\|_{\bm H^1}^2+C\|\bm{u}_\ep^{m}\|_{\bm H^1}\|\bm{u}_\ep^{m\prime}\|_{\bm L^2}^2,
    \end{split} \\
    \begin{split} \label{eq:term4}
      \int_\Gamma g'\alpha_\ep(\bm{u}_\ep^m)\cdot \bm{u}_\ep^{m\prime} ~ds &\leq\|g'\|_{L^2(\Gamma)} \| \bm{u}_\ep^{m\prime} \|_{\bm L^2(\Gamma)}  \quad (\text{ by }|\alpha_\ep(\bm z)| \le 1)\\
      &\leq C\|g'\|_{L^2(\Gamma)}\| \bm{u}_\ep^{m\prime} \|_{\bm H^1} \leq \frac{c}{6}\| \bm{u}_\ep^{m\prime} \|_{\bm H^1}^2+C\|g'\|_{L^2(\Gamma)}^2, 
    \end{split}\\ 
    ( \bm f',\bm{u}_\ep^{m\prime}) &\leq \| \bm f'\|_{L^2}\|\bm{u}_\ep^{m\prime}\|_{\bm H^1} \leq \frac{c}{6}\|\bm{u}_\ep^{m\prime}\|_{\bm H^1}^2+C\| \bm f'\|^2.\label{eq:term6}
  \end{align}
\end{subequations}
Inserting \eqref{eq:trem3-6} into \eqref{eq:dif-t-1}, together with $\int_{\Gamma} g(\bm{u}_{\ep}^{m \prime})^\top \beta(\bm{u}_{\ep}^m) \bm{u}_{\ep}^{m \prime} d s \geq 0$, we obtain
\begin{equation}\label{eq:dif-t-2}
  \frac{d}{dt}\|\bm{u}_\ep^{m\prime}(t)\|^2+c \|\bm{u}_\ep^{m\prime}\|_{\bm H^1}^2\leq C(\| \bm f'\|^2+\|g'\|^2+\|\bm{u}_\ep^{m\prime}\|^2\|\bm{u}_\ep^{m}\|_{\bm H^1}).
\end{equation}
Neglecting the second term on the left side of \eqref{eq:dif-t-2} and using the Gronwall lemma, we get 
\begin{equation}\label{eq:dif-t-3}
  \|\bm{u}_\ep^{m\prime}\|^2 \leq e^{\int_0^{T}C\|\bm{u}_\ep^{m}\|_{\bm H^1}~dt } \big(\|\bm{u}_\ep^{m\prime}(0)\|^2+C\int_0^{T}(\| \bm f'\|^2+\|g'\|^2)~dt \big)
\end{equation}
Next, we consider the boundedness of  $\|\bm{u}_\ep^{m\prime}(0)\|_{L^2}^2$. Multiplying \eqref{eq:w-VEem-a} by $c_k'(t) $ and summing for $k=1,\ldots, m$, then taking $t=0$, we get 
\begin{equation}\label{eq:est-u0-2}
  \|\bm{u}_{\ep}^{m\prime}(0)\|^2+a_1(\bm{u}_{0\ep},\bm{u}_{0\ep} ,\bm{u}_{\ep}^{m\prime}(0))+a_0(\bm{u}_{0\ep} ,\bm{u}_{\ep}^{m\prime}(0))+\int_\Gamma g(0)\alpha_\ep(\bm{u}_{0\ep})\cdot \bm{u}_{\ep}^{m\prime}(0)~ds =( \bm f(0),\bm{u}_{\ep}^{m\prime}(0)).
\end{equation}
The second term on the left-hand of \eqref{eq:est-u0-2} is calculated as follows:
\begin{equation}\label{eq:est-term2}
  \begin{aligned}
    |a_1(\bm{u}_{0\ep},\bm{u}_{0\ep},\bm{u}_{\ep}^{m\prime}(0))| & \le C\|\bm{u}_{0\ep} \|_{\bm L^\infty} \|\bm{u}_{0\ep} \|_{\bm H^1} \|\bm{u}_\ep^{m\prime}(0)\| \\
    &\le C\|\bm{u}_{0\ep} \|_{\bm H^2}^2 \|\bm{u}_\ep^{m\prime}(0)\| \quad(\text{by Sobolev inequality} ) \\
    & \le C(\|\bm h\|+\|g(0)\|_{H^1(\Gamma)})^2\|\bm{u}_\ep^{m\prime}(0)\| \quad (\text{by } ~\eqref{eq:u0e-H2}).
  \end{aligned}
\end{equation}
Combining \eqref{eq:prf-S-wp-III-3'}, \eqref{eq:est-term2} with \eqref{eq:est-u0-2}, we have 
\[
  \|\bm{u}_\ep^{m\prime}(0)\| \le \| \bm f(0)\|+\|\bm h\|+C(\|\bm h\|+\|g(0)\|_{H^1(\Gamma)})^2,
\]
which results in the boundedness of $\|\bm{u}_\ep^{m\prime}(0)\| $. Thus, we conclude $\bm{u}_\ep^{m\prime} \in L^\infty(0, T; \bm H_\sigma)$ from \eqref{eq:dif-t-3}.

%%%
In addition, \eqref{eq:NS-est-0} can be re-written as 
\[
  c\|\bm{u}_\ep^{m}(t)\|_{\bm H^1}^2 \le \frac{1}{c}\| \bm f\|^2- 2(\bm{u}_\ep^{m\prime}, \bm{u}_\ep^{m})  \le \frac{1}{c}\| \bm f\|^2 + 2\|\bm{u}_\ep^{m\prime}\| \|\bm{u}_\ep^{m}\|,
\]
which implies that $\|\bm{u}_\ep^{m}\|_{L^\infty}(0,T;\bm V_\sigma) $ is bounded by $C( \bm f,g,\bm{u}_0)$.

Integrating \eqref{eq:dif-t-2} over $[0,T]$, we get 
\[
  \begin{aligned}
  & \sup_{t \in [0,T]}\|\bm{u}_\ep^{m\prime}(t)\|^2+ c \int_0^T\|\bm{u}_\ep^{m\prime}\|_{\bm H^1}^2 ~dt  \le C\int_0^T \big( \| \bm f'\|^2+\|g'\|^2 + \|\bm{u}_\ep^{m\prime}\|^2\|\bm{u}_\ep^{m}\|_{\bm H^1} \big)~dt \\ 
  & \quad \le C\big(\|\bm f\|^2_{H^1(0, T; \bm L^2)} + \|g '\|^2_{H^1(0,T; L^2(\Gamma))} + \|\bm{u}_\ep^{m\prime}\|_{L^2(0, T; \bm L^2)} + \|\bm{u}_\ep^{m}\|_{L^\infty(0, T; \bm H^1)}\big),
  \end{aligned}
\]
which results in $\bm{u}_\ep^{m\prime} \in L^2(0,T;\bm V_\sigma)$.

%%%
The discussion before \eqref{eq:term3} and the boundedness of $\|\bm{u}_\ep^{m\prime}(0)\|$ hold for both 2D and 3D cases. We re-calculate \eqref{eq:term3} for $d=3$ as follows:
\[
  \begin{aligned}
      a_1(\bm{u}_\ep^{m\prime},\bm{u}_\ep^m ,\bm{u}_\ep^{m\prime}) &\leq C\|\bm{u}_\ep^{m\prime}\|_{\bm L^2}^{\frac{1}{2}}\|\bm{u}_\ep^{m\prime}\|_{\bm H^1}^{\frac{3}{2}}\|\bm{u}_\ep^{m}\|_{\bm H^1} \quad (\text{by }\eqref{eq:a1-3d}) \\
      &\leq \gamma \|\bm{u}_\ep^m\|_{\bm H^1}\|\bm{u}_\ep^{m\prime}\|_{\bm H^1}^2+C\|\bm{u}_\ep^{m}\|_{\bm H^1}\|\bm{u}_\ep^{m\prime}\|^2. 
  \end{aligned}
\]
We choose $\gamma$ small enough such that $\gamma\|\bm{u}_{0\ep}\|_{\bm H^1} \le \frac{c}{12}$. Let $T_1= \max\{ t:\gamma\|\bm{u}_\ep^m (t)\|_{H^1} \le \frac{c}{6}\}$. We have
\[
  \begin{aligned}
    \frac{c}{12 \gamma} &\le \|\bm{u}_\ep^m (T_1)\|_{\bm H^1}-\|\bm{u}_{0\ep}\|_{\bm H^1} \le \|\bm{u}_\ep^m (T_1)-\bm{u}_{0\ep}\|_{\bm H^1} =\|\int_0^{T_1} \bm{u}_\ep^{m\prime} (t) ~dt \|_{\bm H^1} \\
    &\le \int_0^{T_1} \| \bm{u}_\ep^{m\prime} (t) \|_{\bm H^1}~dt  \le \sqrt{T_1} \|\bm{u}_\ep^{m\prime} \|_{L^2(0,T;\bm H^1)}.
  \end{aligned}
\]
It means that $T_1$ has a lower bound independently of $m$ and $\ep$.
Thus, \eqref{eq:dif-t-2} holds for $[0,T_1]$, which yields the boundedness of $\|\bm{u}_\ep^{m\prime}\|_{L^2(0,T_1; \bm V_\sigma) \cap L^\infty (0,T_1; \bm L^2)}$ and $\|\bm{u}_\ep^m\|_{L^\infty(0,T_1; \bm V_\sigma)}$.

%%%
Finally, we pass to limits $m \to \infty$ and $\ep \to 0$ as in Theorem  \ref{th:NS-wp-ue} and \ref{th:NS-wp-u}. Hence, $\bm u_\ep$(resp. $\bm u$) solves {\bf(NS-VE)$_{\sigma \ep}$} (resp. {\bf(NS-VI)$_\sigma$}).

%%%
As in the 2D weak solution, \eqref{eq:2d-uni-1} still holds. The difference is that we estimate the nonlinear term $a_1$ using \eqref{eq:a1-3d} 
\[
  \frac{1}{2}\frac{d}{dt} \|\bm w\|^2 +c \|\bm w\|_{\bm H^1}^2 \le C \|\bm w\|^{\frac{1}{2}} \|\bm w\|_{\bm H^1}^{\frac{3}{2}}\|\bm{u}_1\|_{\bm H^1}  \le c \|\bm w\|_{\bm H^1}^2+ C\|\bm w\|^2\|\bm{u}_1\|_{\bm H^1}^4.
\]
By Gronwall's lemma, we obtain
\[
  \|\bm w(t)\|^2 \le e^{\int_0^T C\|\bm{u}_1\|_{\bm H^1}^4~dt}  \|\bm {w}(0)\|^2.
\]
Since $\bm w(0)=0$ and $\bm{u}_1 \in L^\infty (0, T; \bm V_\sigma)$, we conclude $\bm{u}_1(t) = \bm{u}_2(t)$ for a.e. $t \in [0,T]$.

\end{proof}
%%%

%%%
\begin{proposition}\label{th:wp-p}
  Let $\bm{u}$ be the strong solution of {\bf(NS-VI)$_\sigma$}, and there exists an associated pressure $p \in L^\infty(0,T;\mr{Q} )$ such that $(\bm u, p)$ solves \eqref{eq:NS-P}. 
\end{proposition}
%%%

\begin{proof}
Apply the inf-sup condition \eqref{eq:is1} to \eqref{eq:def-sign}, together with \eqref{a1-2d-3d}, we have
\[
  \begin{aligned}
    C_1\|p\|_Q \le& \sup_{\bm{v}\in \bm V} \frac{b(\bm v,p)}{\|\bm{v}\|_{\bm H^1(\Omega)}} \\ 
    =& \sup_{\bm{v}\in \bm V} \frac{( \bm f,\bm{v}) -(\bm{u}',\bm{v})-a_1(\bm{u},\bm{u},\bm{v})-a_0(\bm{u},\bm{v})-\langle [\bm \sigma \bm n], \bm v \rangle_{\bm \Lambda_{1/2}}}{\|\bm{v}\|_{\bm H^1(\Omega)}}\\ 
    \le& C(\|\bm f\| +\|\bm{u}'\| + \|\bm{u}\|^2_{\bm H^1}+\|\bm{u}\|_{\bm H^1} +\|g\|_{L^2(\Gamma)}),
  \end{aligned}
\]
which is uniformly bounded in $t$. Thus we conclude that $p \in L^\infty(0,T;\mr{Q} )$.
\end{proof}

\bibliographystyle{plain}
\normalem
\bibliography{reference}

\newpage
\appendix
\section{proof of Lemma~\ref{la:est-u0}}\label{App:1}
%%%
\begin{proof}
The proof is divided into three steps. We prove the unique existence of $\bm u_0, \bm u_{0\ep}$ and $p_0, p_{0\ep}$ in (Step 1) . 
The approach is similar to \cite{zhou23},  where the friction leak interface condition \eqref{eq:leak} of the Stokes problem is considered.
Then, we derive the convergence of $\bm u \to \bm u_{0\ep}$ in (Step 2).  Finally, (Step 3) gives the $H^2$-regularity of $\bm u_0$ and $\bm u_{0\ep}$.

%%%
{(Step 1) \bf The unique existence of $\bm u_0$ and $\bm u_{0\ep}$.}
Substituting $\bm v=0, 2\bm u_0$ into \eqref{eq:VI_u0}, together with ~\eqref{eq:a0} and $g(0) >0 $ on $\Gamma$, we have
\[
c\|\bm u_0\|_{\bm H^1}^2 \le a_0(\bm u_0,\bm u_0) +\int_\Gamma g(0)|\bm u_0| ~ds =(\bm h,\bm u_0) \le \|\bm h\| \|\bm u_0\| \le C  \|\bm h\| \|\bm u_0\|_{\bm H^1},
\] 
which implies $ \|\bm u_0\|_{\bm H^1} \le C \|\bm h \|$.

Substituting $\bm v=\bm u_{0\ep}$ into \eqref{eq:VI_u0e}, in views of \eqref{eq:alpha_ep}, we get
\[
c\|\bm u_{0\ep}\|_{\bm H^1}^2 \le a_0(\bm u_{0\ep},\bm u_{0\ep}) +\int_\Gamma g(0)\alpha_\ep(\bm u_{0\ep}) \cdot \bm u_{0\ep} ~ds =(\bm h,\bm u_{0\ep}) \le C  \|\bm h\| \|\bm u_{0\ep}\|_{\bm H^1}.
\] 
Therefore, $\|\bm u_{0\ep}\|_{\bm H^1} \le C \|\bm h \|$.

To show the existence of $p_0$ and $p_{0\ep}$, we apply a similar argument to the proof of Theorem~\ref{th:equi-S}.
By the inf-sup condition ~\eqref{eq:is1}, there exist unique 
\[
(\mr{p}_{0,\mathrm{in}}, \mr{p}_{0,\mathrm{out}}) =: \mr{p}_0,\quad  (\mr{p}_{0\ep,\mathrm{in}}, \mr{p}_{0\ep,\mathrm{out}}) =: \mr{p}_{0\ep} \in L_0^2(\Omega_{\mathrm{in}}) \times L_0^2(\Omega_{\mathrm{out}})
\]
such that, for all $\bm v\in \bm V^0$
\[
\begin{aligned}
    & a_0(\bm u_0,\bm v) + b(\bm v,\mr{p}_0 ) = (\bm h,\bm v), \quad && \|\mr{p}_0\|_{L^2} \le C (\|\bm u_0\|_{\bm H^1} + \|\bm h\|), \\
    & a_0(\bm u_{0\ep},\bm v) + b(\bm v,\mr{p}_{0\ep} ) = (\bm h,\bm v), \quad && \|\mr{p}_{0\ep}\|_{L^2} \le C (\|\bm u_{0\ep}\|_{\bm H^1} + \|\bm h\|).
\end{aligned}   
\]
By duality between $\bm \Lambda_{1/2}$ and $\bm \Lambda_{-1/2}$, there exist unique $\mr{\bm \lambda}_0, \mr{\bm \lambda}_{0\ep} \in \bm \Lambda_{-1/2}$ such that, for all $\bm v \in \bm V$
\[
\begin{aligned}
    & a_0(\bm u_0,\bm v) + b(\bm v,\mr{p}_0 ) +\langle \mr{\bm \lambda}_0,\bm v\rangle_{\bm \Lambda_{1/2}} =(\bm h,\bm v), \quad && \|\mr{\bm \lambda}_0\|_{\bm \Lambda_{-1/2}} \le C (\|\bm u_0\|_{\bm H^1} +\|\mr{p}_0\| + \|\bm h\|), \\ 
    & a_0(\bm u_{0\ep},\bm v) + b(\bm v,\mr{p}_{0\ep} ) + \langle \mr{\bm \lambda}_{0\ep},\bm v\rangle_{\bm \Lambda_{1/2}} =(\bm h,\bm v), \quad && \|\mr{\bm \lambda}_{0\ep}\|_{\bm \Lambda_{-1/2}} \le C (\|\bm u_{0\ep}\|_{\bm H^1} +\|\mr{p}_{0\ep}\|+ \|\bm h\|).
\end{aligned}   
\]

According to \eqref{eq:VI_u0} and \eqref{eq:VE_u0e}, we have 
\begin{subequations}\label{eq:lam}
\begin{align}
    & \langle \mr{\bm \lambda}_0,\bm v-\bm u_0\rangle_{\bm \Lambda_{1/2}} \le \int_\Gamma g(0)|\bm v|~ds -\int_\Gamma g(0) |\bm u_0|~ds  \quad \forall \bm v \in \bm V_\sigma, \label{eq:lam-a} \\ 
    & \langle \mr{\bm \lambda}_{0\ep},\bm v \rangle_{\bm \Lambda_{1/2}} = \int_\Gamma g(0)\alpha_\ep(\bm u_{0\ep})~ds  \quad \forall \bm v \in \bm V_\sigma.  \label{eq:lam-b}
\end{align}
\end{subequations}
Applying Hahn-Banach theorem, there exists a $\mr{\bm \lambda}_0 \in (\bm L_{g(0)}^1 (\Gamma))' = \bm L_{1/g(0)}^\infty (\Gamma)$ such that 
\[
|\langle \mr{\bm \lambda}_0, \bm \mu \rangle_{\bm L_{g(0)}^1 (\Gamma)}| \le \int_\Gamma g(0)|\bm \mu|~ds  \quad \forall \bm \mu \in \bm L_{g(0)}^1 (\Gamma).
\]
Since $\mr{\bm \lambda}_0 -\bm \lambda_0$ vanishes on $\mr{\bm \Lambda}_{1/2}$, there exists a constant $\delta$, such that $\mr{\bm \lambda}_0 -\bm \lambda_0 = \delta \bm{n}$.
Given $\delta$ satisfying $|\mr{\bm \lambda}_0 - \delta \bm{n}| \le g(0)$ [a.e.] on $\Gamma$, we can find unique constant pair $(k_{0,\mathrm{in}}, k_{0,\mathrm{out}})$ such that 
\begin{equation}\label{eq:k}
k_{0,\mathrm{out}}-k_{0,\mathrm{in}} = \delta, \quad |\Omega_{\mathrm{in}}|k_{0,\mathrm{in}} + |\Omega_{\mathrm{out}}|k_{0,\mathrm{out}} = 0.
\end{equation}
Set $p_0 := (\mr{p}_{0,\mathrm{in}}+ k_{0,\mathrm{in}}, \mr{p}_{0,\mathrm{out}}+ k_{0,\mathrm{out}} )$, we find that 
\begin{equation}\label{eq:p0}
a_0(\bm u_0, \bm v) + b(\bm v, p_0) +\langle \bm \lambda_0, \bm v\rangle_{\bm \Lambda_{1/2}} =(\bm h, \bm v) \quad \forall \bm v\in \bm V.
\end{equation}
Since 
\[
|\bm \lambda_0| \le g(0), \quad \int_\Gamma \bm \lambda_0 \bm u_0~ds = \int_\Gamma  \mr{\bm \lambda}_0 \bm u_0~ds = \int_\Gamma g(0) |\bm u_0|~ds,
\] 
we conclude the $(\bm u_0, p_0)$ solves 
\[
a_0(\bm u_0, \bm v-\bm u_0) + b(\bm v-\bm u_0, p_0) + \int_\Gamma g(0) |\bm v|~ds-\int_\Gamma g(0) |\bm u_0|~ds \ge (\bm h, \bm v - \bm u_0) \quad \forall \bm v\in \bm V.
\]
It's easy to see that 
\[
\langle \bm \lambda_0, \bm v\rangle \le \int_\Gamma g(0) |\bm v|~ds, \quad \langle \bm \lambda_0, \bm u_0 \rangle = \int_\Gamma g(0) |\bm u_0|~ds,
\]
which implies $\bm \lambda_0 \le g(0)$ and  $\bm \lambda_0 \bm u_0 = g(0)\bm u_0$.

If $\{ x \in \Gamma \mid \bm u_0 \cdot \bm n \neq 0\}$ has a positive measure, we assume that there is another constant $\bar{\delta}$ such that 
$\mr{ \bm \lambda}_0 +\bar{\delta}\bm n$ satisfies 
\[
(\mr{ \bm \lambda}_0 +\bar{\delta}\bm n)\cdot \bm u_0 =g(0)|\bm u_0|, \quad |\mr{ \bm \lambda}_0 +\bar{\delta}\bm n| \le g(0).
\]
We see that 
\[
(\bar{\delta}-\delta)\bm n \cdot \bm u_0 = (\mr{ \bm \lambda}_0 +\bar{\delta}\bm n)\cdot \bm u_0- (\mr{ \bm \lambda}_0 + \delta\bm n)\cdot \bm u_0 = g(0)|\bm u_0| -g(0)|\bm u_0| =0,
\]
which implies $\bar{\delta} =\delta$. Hence $\delta$ is unique, so do $\bm \lambda_0$ and $p_0$.

%%%
If $\bm u_0\cdot \bm n=0$  on $\Gamma$, then, any $\delta$ satisfies $ |\mr{ \bm \lambda}_0 - \delta\bm n| \le g(0)$ [a.e.] on $\Gamma$, we have $(\bm \lambda_0, p_0)$ solves \eqref{eq:p0}.
In fact, we have 
\[
\begin{aligned}    
|\Gamma||\delta| & = \langle \delta \bm n, \bm n\rangle= \langle \bm \lambda_0- \mr{ \bm \lambda}_0, \bm n\rangle = \langle \bm \lambda_0, \bm n\rangle - \langle \mr{ \bm \lambda}_0, \bm n\rangle \le \int_\Gamma g(0)~ds +|\langle \mr{ \bm \lambda}_0, \bm n\rangle| \\
& \le \int_\Gamma g(0)~ds + \|\mr{ \bm \lambda}_0\|_{\bm \Lambda_{-1/2}} \|\bm n\|_{\bm \Lambda_{1/2}} \le \|g(0)\|_{L^1(\Gamma)} + C \|\mr{ \bm \lambda}_0\|_{\bm \Lambda_{-1/2}} .
\end{aligned} 
\]
As a result, we obtain 
\[
\|p_0\|_{L^2} \le C(\|\bm u_0\|_{\bm V} +\|\bm h\| +\|g(0)\|_{L^1(\Gamma)}).
\]

\eqref{eq:lam-b} implies that there exists a $\delta \in \mathbb{R}$ such that $\mr{\bm \lambda}_{0 \ep} +\delta \bm n = g(0) \alpha_\ep(\bm u_{0\ep}) =:\bm\lambda_{0\ep}$ on $\Gamma$.
Then we can determine the constant pair $(k_{0,\mathrm{out}}, k_{0,\mathrm{in}})$ satisfying \eqref{eq:k}. As a result, we have 
\[
a_0(\bm u_{0\ep},\bm v) + b(\bm v,p_{0\ep} ) +\int_\Gamma g(0)\alpha_\ep(\bm u_{0\ep}) \bm v ~ds = (\bm h,\bm v) \quad \forall \bm v\in \bm V,
\]
where $p_{0\ep} = \mr{p}_{0\ep} + (k_{0,\mathrm{out}}, k_{0,\mathrm{in}})$. Again, we have (by $|\alpha_\ep(\bm z)| \le 1$) 
\[
\begin{aligned}    
|\Gamma||\delta| & = \langle \delta \bm n, \bm n\rangle= \langle \bm \lambda_{0\ep}- \mr{ \bm \lambda}_{0\ep}, \bm n\rangle = \langle \bm \lambda_{0\ep}, \bm n\rangle - \langle \mr{ \bm \lambda}_{0\ep}, \bm n\rangle 
= \langle g(0) \alpha_\ep(\bm u_{0\ep}), \bm n \rangle  - \langle \mr{ \bm \lambda}_{0\ep}, \bm n\rangle\\
& \le \int_\Gamma g(0)~ds + \|\mr{ \bm \lambda}_{0\ep}\|_{\bm \Lambda_{-1/2}} \|\bm n\|_{\bm \Lambda_{1/2}} \le \|g(0)\|_{L^1(\Gamma)} + C \|\mr{ \bm \lambda}_{0\ep}\|_{\bm \Lambda_{-1/2}} .
\end{aligned} 
\]
Hence, 
\[
\|p_{0\ep}\|_{L^2} \le C(\|\bm u_0\|_{\bm V} +\|\bm h\| +\|g(0)\|_{L^1(\Gamma)}).
\] 

{(Step 2) \bf The convergence of $\bm u_0 \to \bm u_{0\ep}$.}
Substituting $\bm v =\bm u_{0\ep} $ into \eqref{eq:VI_u0} and  $\bm v =\bm u_0$ into \eqref{eq:VI_u0e}, then adding the resulted inequalities together, we observe that
\[
\begin{aligned}
    a_0(\bm u_0-\bm u_{0\ep}, \bm u_{0\ep}-\bm u_0) & \le \int_\Gamma g(0) |\bm u_{0\ep}| ~ds- \int_\Gamma g(0) |\bm u_0| ~ds + \int_\Gamma g(0) \rho_\ep(\bm u_0)~ds  \\ 
    & \quad - \int_\Gamma g(0) \rho_\ep(\bm u_{0\ep}) ~ds  \le 2\ep \int_\Gamma g(0)~ds \le C\ep.
\end{aligned}
\]
As a result, we have 
\begin{equation}\label{eq:u0e-u0-H1}
\|\bm u_0-\bm u_{0\ep} \|_{\bm H^1} \le C\sqrt{\ep} \to 0 \text{ as }\ep \to 0.
\end{equation}

%%%%
{(Step 3) \bf The $H^2$-regularity of $\bm u_0$ and $\bm u_{0\ep}$.}
Next, we prove that $\bm u_{0\ep} \in \bm H^2(\Omega_a)$. We can find a partition of $\Omega$ (see Fig.~\ref{fig:domain_1}), denoted by 
$\{\Omega_{1i}, \Omega_2, \Omega_{3j}, \Omega_4: ~ i=1,2,\ldots,m; ~ j=1,2,\ldots,n\}$, where $\Omega_{1i}, \Omega_2, \Omega_{3j}, \Omega_4$ are open, smooth, single-connect subset of $\mathbb{R}^d$ satisfying
\[
\begin{aligned}
    & \overline{\Omega}_2 \subsetneqq \Omega_{\mathrm{out}}, \quad \overline{\Omega}_4 \subsetneqq \Omega_{\mathrm{in}}, \quad \Omega_{1i} \cap \Omega_{3j} = \emptyset (\forall i,j),\quad  \Omega_{1i} \cap \Gamma_D =\emptyset, \quad  \Omega_{3j}\cap \Gamma = \emptyset, \\ 
    & (\bigcup_{i=1}^m \Omega_{1i} \bigcup_{j=1}^n \Omega_{1j}  \cup \Omega_2 \cup \Omega_4) \supset (\overline{ \Omega_{\mathrm{in}} \cup \Omega_{\mathrm{out}} \cup \Gamma}).
\end{aligned}
\]
\begin{figure}[htbp] 
\centering
\begin{overpic}[scale=0.40]{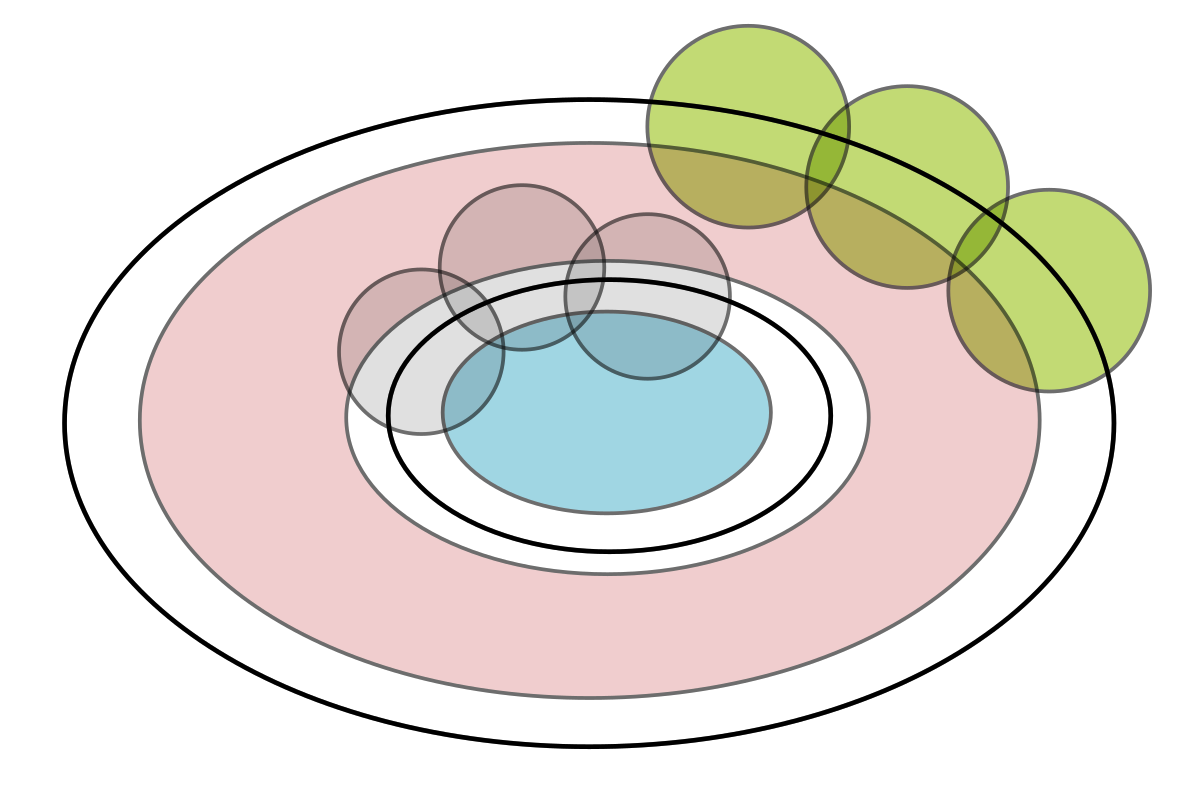}
    \put(73,49){$\Omega_{11}$}
    \put(85,40){$\Omega_{12}$}
    \put(91,31){$\ldots$}
    \put(58,55){$\Omega_{1m}$}
    \put(48,59){$\ldots$}
    \put(16,30){$\Omega_2$}
    \put(40,43){$\Omega_{31}$}
    \put(51,40){$\Omega_{32}$}
    \put(63,37){$\ldots$}
    \put(31,35){$\Omega_{3n}$}
    \put(30,28){$\ldots$}
    \put(47,30){$\Omega_4$}
    \put(66,24){$\Gamma$}
    \put(66,4){$\Gamma_D$}
\end{overpic}
\caption{The partition of  $\Omega_{\mathrm{in}} \cup \Gamma \cup \Omega_{\mathrm{out}}$}\label{fig:domain_1}
\end{figure}

The $H^2/H^1$-regularity of $(\bm u_{0\ep}, p_{0\ep})$ in $\Omega_{1i} \cap \Omega_{\mathrm{out}}, \Omega_2$ and $\Omega_4$ follows from the standard argument. 
For example, we can find a slightly bigger open set $\widetilde{\Omega}_{1i} \supsetneqq \overline{\Omega}_{1i}$, $\widetilde{\Omega}_{1i} \cap \Omega_{3j} = \emptyset (\forall j)$ and a $\theta_{1i} \in \bm C_0^\infty(\widetilde{\Omega}_{1i})$ with $0 \le \theta_{1i} \le 1$ and $\theta_{1i}=1$ in $\Omega_{1i}$.
Then we can derive the equation of $(\theta_{1i}\bm u_{0\ep}, \theta_{1i}p_{0\ep})$ in $\widetilde{\Omega}_{1i} \cap \Omega_{\mathrm{out}}$ with the Dirichlet boundary condition. 
The $H^2/H^1$-regularity of  $(\theta_{1i}\bm u_{0\ep}, \theta_{1i}p_{0\ep})$ in $\widetilde{\Omega}_{1i} \cap \Omega_{\mathrm{out}}$ follows from the standard argument.
The case of $\Omega_2$ and $\Omega_4$ is similar.

Below, we show that $\bm u_{0\ep} \in \bm H^2((\bigcup_{j}\Omega_{3j})\cap \Omega_{\mathrm{in}})$, $\bm u_{0\ep} \in \bm H^2((\bigcup_{j}\Omega_{3j})\cap \Omega_{\mathrm{out}})$, $p_{0\ep} \in H^1((\bigcup_{j}\Omega_{3j})\cap \Omega_{\mathrm{in}})$, $p_{0\ep} \in H^1((\bigcup_{j}\Omega_{3j})\cap \Omega_{\mathrm{out}})$. 
It suffices to study the regularity near $\Gamma$. Let $x_0 \in \Gamma$ and $U \subset \mathbb R_x^d$ be a neighbourhood of $x_0$. 
There is an open subset $U \subset U_0(x_0\in U)$ and a bijection  $\Phi = (\Phi_1,\ldots,\Phi_d)$ from $U$ to $\tilde U \subset \mathbb R_y^d$ enjoying the following properties
\begin{enumerate}
  \item[1.] $\Phi$ is a $C^3$-diffeomorphism,
  \item[2.] $\Phi(x_0)=0$,
  \item[3.] $\Phi(U\cap \Omega_{\mathrm{in}})= Q_1:=\{ y=(y',y_d)\in \mathbb{R}^d_y,~|y'| <R,~-R < y_d<0\}$, 
  \item[4.] $\Phi(U\cap \Omega_{\mathrm{out}})= Q_2:=\{ (y',y_d)\in \mathbb{R}^d_y,~|y'| <R,~0< y_d<R\},\quad Q:= Q_1 \cup Q_2$,
  \item[5.] $\Phi(U\cap \Gamma)= S:=\{ y=(y',0) \mid |y'| <R \}$, 
  \item[6.] $\frac{\partial \Phi_d}{\partial x_j} =\frac{\partial \Phi_j}{\partial x_d}=0 (j=1,\ldots,d-1),~ \frac{\partial \Phi_d}{\partial x_d}=1 $ on $S$.
\end{enumerate}

For briefness, we omit the subscript $0$. Setting $y= \Phi(x)=(\Phi_1(x),\ldots,\Phi_d(x))$,
\[
  \tilde {\bm u}_\ep (y)=\bm u_{0\ep}(x),~\tilde{p}_\ep(y)=p_{0\ep}(x),~\tilde{\bm h}(y)=\bm h(x),~\tilde{g}(y)=g(x).
\]
Below the summation w.r.t. $i$ or $j$ is omitted for briefness. Set the bilinear forms
\[
\begin{aligned}
& \tilde{a}_0 (\tilde{\bm u}_\ep, \tilde{\bm v}) = 2\nu \int_Q \tilde{e}_{ij}(\tilde{\bm u}_\ep) \tilde{e}_{ij}(\tilde{\bm v}) |\operatorname{Jac} \Phi|^{-1}~dy,\quad \tilde{e}_{ij}(\tilde{\bm v}) = \frac{1}{2}\sum_{k=1}^d (\frac{\partial \tilde{\bm v}_i}{\partial y_k}\frac{\partial \Phi_k}{\partial x_j}+\frac{\partial \tilde{\bm v}_j}{\partial y_k}\frac{\partial \Phi_k}{\partial x_i}), \\ 
& \tilde {b} (\tilde{\bm v}, \tilde{p}_\ep) =  -\int_Q  \tilde{p}_\ep \sum_{k} \frac{\partial \tilde{\bm v}_i}{\partial y_k}\frac{\partial \Phi_k}{\partial x_i} |\operatorname{Jac} \Phi|^{-1} ~dy. 
\end{aligned}
\]
We see that 
\begin{subequations}\label{eq:S-y}
  \begin{align}
    \tilde {a}_0 (\tilde{\bm u}_\ep, \tilde{\bm v}) + \tilde {b} (\tilde{\bm v}, \tilde{p}_\ep) + \int_S \tilde{g}(0) \alpha_\ep(\tilde {\bm u}_\ep)\tilde{\bm v} \sqrt{\mathrm{det}G}~dy' &= \int_Q \tilde{\bm h} \cdot \tilde{\bm v} |\operatorname{Jac} \Phi|^{-1}dy, \label{eq:S-y-1} \\ 
    \tilde {b} (\tilde{\bm u}_\ep, \tilde{q}) &=0, \label{eq:S-y-2}
  \end{align}
\end{subequations}
where the matrix $G = (G_{ij})_{1\le i,j\le d-1}$ denotes the Riemannian metric tensor given by $G_{ij} = \frac{\partial \Phi}{\partial x_i}\cdot \frac{\partial \Phi}{\partial x_j}$ (dot means the inner product in $\mathbb{R}^d$)~\cite{Kashiwabara15}.

Define the difference quotient operator $D_h^l$ by: 
\[
D_h^l \bm v(y) =\frac{S_h^l\bm v(y)-\bm v(y)}{|h|} \quad (h \in \mathbb{R}), \quad S_h^l\bm v(y) = \bm v(y+he_l).
\]
where $\{e_l\}_{l=1}^{d-1}$ are the canonical basis of $\mathbb R_y^d$.
The following facts are well-known
\[
\begin{aligned}
& (\bm u, D_{-h}^l \bm v)_Q = (D_h^l \bm u, \cdot \bm v)_Q, \quad \|D_h^l \bm u\|_{L^2(Q)} \le C\|\nabla_{y_l} \bm u\|_{L^2(Q)}, \\ 
& D_h^l(\bm u \bm v)= \bm u (D_h^l \bm v)+ (D_h^l \bm u)(S_h^l \bm v).
\end{aligned}
\]
Now, we take $\tilde{\bm v} =D_{-h}^l(\theta^2 D_h^l \tilde{\bm u}_\ep)$ in \eqref{eq:S-y-1} where $\theta \in \bm C^\infty(\mathbb R^d_y)$ satisfies
\[
  0 \le \theta \le 1, \quad \operatorname{supp} \theta \subset U, \quad \theta =1 \text{ in } Q_{R/2},\quad \theta = 0 \text{ for } |y'| \ge R \text{ and } |y_d| \ge R.
\]
Each term of \eqref{eq:S-y-1} is estimated as follows: 
\begin{subequations}\label{eq:S-y-est}
\begin{align}
    &\tilde {a}_0 (\tilde{\bm u}_\ep, \tilde{\bm v}) \ge c_1 \|\nabla_y(\theta D_h^l \tilde{\bm u}_\ep)\|^2 - c_2\|\nabla \tilde{\bm u}_\ep\|(\|\nabla(\theta D_h^l\tilde{\bm u}_\ep)\| +\|\nabla \tilde{\bm u}_\ep\|) , \label{eq:S-y-est-a}\\ 
    &|\tilde {b} (\tilde{\bm v}, \tilde{p}_\ep)| \le c_3\|\tilde{p}_\ep\|_Q(\|\nabla(D_h^l\tilde{\bm u}_\ep)\| + \|\nabla\tilde{\bm u}_\ep\|), \label{eq:S-y-est-b} \\ 
    & - \int_S \tilde{g}(0) \alpha_\ep(\tilde {\bm u}_\ep) \cdot \tilde{\bm v}(y',0) \sqrt{\mathrm{det}G}~dy' \le c_4\|\nabla_{y'} g(0)\|_{L^2(S)} \big( \|\tilde{\bm u}_\ep\|_{H^1}+ \|\nabla (\theta D_h^l \tilde{\bm u}_\ep)\| \big), \label{eq:S-y-est-c} \\ 
    & \int_Q \tilde{\bm h} \cdot \tilde{\bm v} |\operatorname{Jac} \Phi|^{-1}dy \le c_5 \|\tilde{\bm h} \|(\|\tilde{\bm u}_\ep\|_{H^1} + \|\nabla (\theta D_h^l \bm u_\ep)\|). \label{eq:S-y-est-d}
\end{align}
\end{subequations}

Next, we will provide a detailed estimation of \eqref{eq:S-y-est}.

%%%%
{\bf{Proof of ~\eqref{eq:S-y-est-a}.}}
\[
\begin{aligned}
    \tilde {a}_0 (\tilde{\bm u}_\ep, \tilde{\bm v}) &= \frac{\nu}{2} \int_Q \sum_{k=1}^d \bigg(\frac{\partial \tilde{\bm u}_{\ep i}}{\partial y_k}\frac{\partial \Phi_k}{\partial x_j}+\frac{\partial \tilde{\bm u}_{\ep j}}{\partial y_k}\frac{\partial \Phi_k}{\partial x_i}\bigg) 
        D_h^{-l}\sum_{k=1}^d \bigg(\frac{\partial (\theta^2 D_h^{-l}\tilde{\bm u}_{\ep i})}{\partial y_k}\frac{\partial \Phi_k}{\partial x_j} \\ 
    & \qquad +\frac{\partial (\theta^2 D_h^{-l}\tilde{\bm u}_{\ep j})}{\partial y_k}\frac{\partial \Phi_k}{\partial x_i} \bigg) |\operatorname{Jac} \Phi|^{-1}~dy \\ 
    & = \frac{\nu}{2} \int_Q D_h^l \bigg(\sum_{k=1}^d (\frac{\partial \tilde{\bm u}_{\ep i}}{\partial y_k}\frac{\partial \Phi_k}{\partial x_j}+\frac{\partial \tilde{\bm u}_{\ep j}}{\partial y_k}\frac{\partial \Phi_k}{\partial x_i})|\operatorname{Jac} \Phi|^{-1}\bigg) 
     \sum_{k=1}^d \bigg((\theta \frac{ \partial{(\theta D_h^l \tilde{\bm u}_{\ep i})}}{\partial y_k} + \theta D_h^l \tilde{\bm u}_{\ep i} \frac{\partial \theta}{ \partial {y_k}} ) \frac{\partial \Phi_k}{\partial x_j} \\ 
    & \qquad + (\theta \frac{ \partial{(\theta D_h^l \tilde{\bm u}_{\ep j})}}{\partial y_k} + \theta D_h^l \tilde{\bm u}_{\ep j} \frac{\partial \theta}{ \partial {y_k}}) \frac{\partial \Phi_k}{\partial x_i}\bigg) ~dy \\ 
    & = \frac{\nu}{2} \int_Q \bigg[ \underbrace
        {\sum_k \theta (\frac{\partial (D_h^l \tilde{\bm u}_{\ep i})}{\partial y_k}\frac{\partial \Phi_k}{\partial x_j} + \frac{\partial (D_h^l\tilde{\bm u}_{\ep j})}{\partial y_k}\frac{\partial \Phi_k}{\partial x_i} )}_{ = 2\tilde{e}_{ij}(\tilde{\bm u}_\ep)- \tilde{c}_{ij}(\tilde{\bm u}_\ep)}
        |\operatorname{Jac} \Phi|^{-1} \\ 
    & \qquad + \underbrace{\sum_k \theta S_h^l \frac{\partial \tilde{\bm u}_{\ep i}}{\partial y_k}D_h^l(\frac{\partial \Phi_k}{\partial x_j}|\operatorname{Jac} \Phi|^{-1})
        +\sum_k \theta S_h^l \frac{\partial \tilde{\bm u}_{\ep j}}{\partial y_k}D_h^l(\frac{\partial \Phi_k}{\partial x_i}|\operatorname{Jac} \Phi|^{-1})}_{\tilde{d}_{ij}(\tilde{\bm u}_\ep)} \bigg] \\
    & \qquad \cdot \bigg[ 2\tilde{e}_{ij}(\theta D_h^l \tilde{\bm u}_\ep) + \underbrace{\sum_k ( D_h^l \tilde{\bm u}_{\ep i} \frac{\partial \theta }{\partial y_k}\frac{\partial \Phi_k}{\partial x_j} + D_h^l \tilde{\bm u}_{\ep j} \frac{\partial \theta }{\partial y_k}\frac{\partial \Phi_k}{\partial x_i})}_{=: \tilde{c}_{ij}(\tilde{\bm u}_\ep)} \bigg] ~dy \\ 
    & = \tilde{a}_0 (\theta D_h^l \tilde{\bm u}_\ep, \theta D_h^l \tilde{\bm u}_\ep) +\frac{\nu }{2} \int_Q \tilde{d}_{ij}(\tilde{\bm u}_\ep)\cdot (2\tilde{e}_{ij}(\theta D_h^l \tilde{\bm u}_\ep)+\tilde{c}_{ij}(\tilde{\bm u}_\ep)) - |\tilde{c}_{ij}(\tilde{\bm u}_\ep)|^2 |\operatorname{Jac}|^{-1} ~dy \\ 
    & \ge c_1 \|\nabla (\theta D_h^l \tilde{\bm u}_\ep )\|^2 - c_2\|\nabla \tilde{\bm u}_\ep \|( \|\nabla (\theta D_h^l \tilde{\bm u}_\ep )\| + \|\nabla \tilde{\bm u}_\ep \|) .
\end{aligned}
\]
%

%%%%
{\bf{Proof of ~\eqref{eq:S-y-est-b}.}} We omit the summation w.r.t. $k$.
Substituting $\tilde{q} = D_{-h}^l (\theta D_h^l \tilde{p}_\ep)$ into \eqref{eq:S-y-2} 
\begin{equation}\label{eq:S-y-est-1}
\begin{aligned}
    0 = &  - \int_Q  D_h^l \tilde{p}_\ep D_h^l(\frac{\partial \tilde{\bm u}_{\ep i}}{\partial y_k} \frac{\partial \Phi_k}{\partial x_i} |\operatorname{Jac} \Phi|^{-1})~dy \\ 
    = & - \int_Q  D_h^l \tilde{p}_\ep (\frac{\partial \Phi_k}{\partial x_i} |\operatorname{Jac} \Phi|^{-1} \theta) \frac{\partial (\theta D_h^l \tilde{\bm u}_{\ep i})}{\partial y_k} + (\frac{\partial \Phi_k}{\partial x_i} |\operatorname{Jac} \Phi|^{-1}) D_h^l \tilde{p}_\ep \frac{\partial \theta}{\partial y_k} \theta D_h^l \tilde{\bm u}_{\ep i} \\ 
    & \qquad + \tilde{p}_\ep D_{-h}^l (\theta^2 S_h^l \frac{\partial \tilde{\bm u}_{\ep i}}{\partial y_k}) D_h^l (\frac{\partial \Phi_k}{\partial x_i} |\operatorname{Jac} \Phi|^{-1} )~dy \\
    = & - \int_Q  D_h^l \tilde{p}_\ep (\frac{\partial \Phi_k}{\partial x_i} |\operatorname{Jac} \Phi|^{-1} \theta) \frac{\partial (\theta D_h^l \tilde{\bm u}_{\ep i})}{\partial y_k} ~dy
    - \int_Q \tilde{p}_\ep D_{-h}^l \bigg[ \theta^2 S_h^l \frac{\partial \tilde{\bm u}_{\ep i}}{\partial y_k} D_h^l ( \frac{\partial \Phi_k}{\partial x_i} |\operatorname{Jac} \Phi|^{-1}) \\ 
    & \qquad + \frac{\partial \Phi_k}{\partial x_i} |\operatorname{Jac} \Phi|^{-1} \frac{\partial \theta}{\partial y_k} \theta D_h^l \tilde{\bm u}_{\ep i} \bigg]~dy
\end{aligned}
\end{equation}
\begin{equation}\label{eq:S-y-est-2}
\begin{aligned}
    \tilde b(\tilde{\bm v}, \tilde{ p}_\ep) & =  -\int_Q  \tilde{p}_\ep D_{-h}^l  \frac{\partial (\theta^2 D_h^l \tilde{\bm u}_{\ep i})}{\partial y_k}\frac{\partial \Phi_k}{\partial x_i} |\operatorname{Jac} \Phi|^{-1} ~dy  \\ 
    & =   -\int_Q  \tilde{p}_\ep \theta  D_{-h}^l (\frac{\partial (\theta D_h^l \tilde{\bm u}_{\ep i})}{\partial y_k}) \frac{\partial \Phi_k}{\partial x_i} |\operatorname{Jac} \Phi|^{-1} ~dy 
        - \int_Q \tilde{p}_\ep \bigg[ (D_{-h}^l \theta) S_{-h}^l (\frac{\partial (\theta D_h^l \tilde{\bm u}_{\ep i})}{\partial y_k}) \\ 
    & \qquad +( D_{-h}^l D_h^l \tilde{\bm u}_{\ep i} ) \frac{\partial \theta}{ \partial y_k}\theta  + D_{-h}^l ( \frac{\partial \theta}{ \partial y_k}\theta) S_{-h}^l D_h^l \tilde{\bm u}_{\ep i} \bigg] \frac{\partial \Phi_k}{\partial x_i} |\operatorname{Jac} \Phi|^{-1} ~dy \\ 
    & = -\int_Q D_h^l \tilde{p}_\ep (\frac{\partial \Phi_k}{\partial x_i} |\operatorname{Jac} \Phi|^{-1} \theta) \frac{\partial \theta D_h^l \tilde{\bm u}_{\ep i}}{ \partial y_k} ~ dy - 
        \int_Q S_h^l \tilde{p}_\ep D_h^l(\frac{\partial \Phi_k}{\partial x_i} |\operatorname{Jac} \Phi|^{-1} \theta)\frac{\partial \theta D_h^l \tilde{\bm u}_{\ep i}}{ \partial y_k} ~ dy \\ 
    & \qquad  - \int_Q \tilde{p}_\ep \bigg[ (D_{-h}^l \theta) S_{-h}^l (\frac{\partial (\theta D_h^l \tilde{\bm u}_\ep i)}{\partial y_k}) 
        +( D_{-h}^l D_h^l \tilde{\bm u}_{\ep i} ) \frac{\partial \theta}{ \partial y_k}\theta  \\ 
    & \qquad + D_{-h}^l ( \frac{\partial \theta}{ \partial y_k}\theta) S_{-h}^l D_h^l \tilde{\bm u}_{\ep i} \bigg] \frac{\partial \Phi_k}{\partial x_i} |\operatorname{Jac} \Phi|^{-1} ~dy.
\end{aligned}
\end{equation}
Summing up \eqref{eq:S-y-est-1} and \eqref{eq:S-y-est-2}, we get 
\[
|\tilde b(\tilde{\bm v}, \tilde{p}_\ep) | \le c_3\|\tilde{p}_\ep\| (\|\nabla (\theta D_h^l \tilde{\bm u}_\ep)\| +\|\nabla \tilde{\bm u}_\ep \|). 
\]
%

%%%%
{\bf{Proof of ~\eqref{eq:S-y-est-c}.}}
\[
\begin{aligned}
    & -\int_S \tilde{g}(0) \alpha_\ep (\tilde{\bm u}_\ep) \cdot D_h^l (\theta^2 D_h^l \tilde{\bm u}_\ep(y',0)) \sqrt{\mathrm{det}G} ~dy' \\
    = & -\int_S \theta^2  \tilde{g}(0) \underbrace{D_h^l \alpha_\ep (\tilde{\bm u}_\ep) \cdot D_h^l(\tilde{\bm u}_\ep)}_{\ge 0(\text{by } \eqref{eq:alpha_ep})} \sqrt{\mathrm{det}G} ~dy' 
        -\int_S \theta^2 S_h^l (\alpha_\ep(\tilde{\bm u}_\ep) \cdot D_h^l(\tilde{\bm u}_\ep)) D_h^l (\tilde{g}(0)\sqrt{\mathrm{det}G}) ~dy' \\ 
    \le & \int_S \theta^2 \big|D_h^l(\tilde{\bm u}_\ep) \big| \big| D_h^l (\tilde{g}(0)\sqrt{\mathrm{det}G}) \big|~dy' \le c_4\| \nabla_{y'}\tilde{g}(0)\|_{L^2(S)}\|\theta D_h^l \tilde{\bm u}_\ep\|_{L^2(S)} \\ 
    \le & c_4\| \nabla_{y'}\tilde{g}(0)\|_{L^2(S)}\|\theta D_h^l \tilde{\bm u}_\ep\|^{\frac{1}{2}} \|\nabla(\theta D_h^l \tilde{\bm u}_\ep )\|^{\frac{1}{2}}  \le  c_4 \|\tilde{g}(0)\|_{H^1(S)} \big( \|\tilde{\bm u}_\ep\|_{\bm H^1} + \|\nabla(\theta D_h^l \tilde{\bm u}_\ep )\| \big).
\end{aligned}
\]
Inserting the estimates \eqref{eq:S-y-est-a}-\eqref{eq:S-y-est-d} into \eqref{eq:S-y-1}, we conclude 
\[
\|\nabla_y (\theta D_h^l \tilde{\bm u}_\ep) \|_{L^2(Q)} \le C (\|\tilde{\bm u}_\ep\|_{\bm H^1} + \|\tilde{p}_\ep\| + \|\tilde{g}\|_{H^1(S)} + \|\tilde{\bm h}\|).
\]
Passing to the limits for $h \to 0$, 
\[
\sum_{i=1}^{d-1}\sum_{j=1}^{d} \|\frac{\partial}{\partial y_i}(\theta \frac{\partial \tilde{\bm u}_\ep}{\partial y_j})\|_{L^2(Q)} \le C (\|\tilde{\bm u}_\ep\|_{\bm H^1} + \|\tilde{p}_\ep\| + \|\tilde{g}\|_{H^1(S)} + \|\tilde{\bm h}\|), 
\]
which implies 
\[
\sum_{i=1}^{d-1}\sum_{j=1}^{d} \|\frac{\partial}{\partial y_i} \frac{\partial \tilde{\bm u}_\ep}{\partial y_j} \|_{L^2(Q_{R/2})} \le C (\|\tilde{\bm u}_\ep\|_{\bm H^1} + \|\tilde{p}_\ep\| + \|\tilde{g}\|_{H^1(S)} + \|\tilde{\bm h}\|).
\]
Together with the trace theorem, we have 
\[
\sum_{j=1}^d \|\frac{\partial \tilde{\bm u}_{\ep,d}}{\partial y_i} |_{Q_1}\|_{H^{\frac{1}{2}}(S)} + \sum_{j=1}^d \|\frac{\partial \tilde{\bm u}_{\ep,d}}{\partial y_i} |_{Q_2}\|_{H^{\frac{1}{2}}(S)} \le C (\|\tilde{\bm u}_\ep\|_{\bm H^1} + \|\tilde{p}_\ep\| + \|\tilde{g}\|_{H^1(S)} + \|\tilde{\bm h}\|), 
\]
which means all tangential derivatives of $\tilde{\bm u}_{\ep,d}|_{Q_1}, \tilde{\bm u}_{\ep,d}|_{Q_2} \in H^{\frac{1}{2}}(S)$. Therefore, we have $\tilde{\bm u}_{\ep,d}|_{Q_1}, \tilde{\bm u}_{\ep,d}|_{Q_2} \in H^{\frac{3}{2}}(S)$. Combining the above estimates with the argument of the partition of unity, we get
\[
\begin{aligned}
    \|\bm u_{0\ep,\bm n} |_{\Omega_{\mathrm{in}}} \|_{\bm H^{\frac{3}{2}}(\Gamma)}  +\|\bm u_{0\ep, \bm n} |_{\Omega_{\mathrm{out}}} \|_{\bm H^{\frac{3}{2}}(\Gamma)} & \le C(\|\bm u_{0\ep}\|_{\bm H^1}+\|p_{0\ep}\|+\|g(0)\|_{H^1(\Gamma)}+\|\bm h\|) \\ 
    & \le C(\|g(0)\|_{H^1(\Gamma)}+\|\bm h\|).
\end{aligned}
\]
Hence, we find \eqref{eq:u0e-H2}.

Since the $H^2$-norm of $\bm u_{0\ep}|_{\Omega_{\mathrm{in}}}$ and $\bm u_{0\ep}|_{\Omega_{\mathrm{out}}}$ is bounded independent of $\ep$, with \eqref{eq:u0e-u0-H1}, there exists a subsequence such that 
\[
\bm u_{0\ep} |_{\Omega_{\mathrm{in}}} \rightharpoonup \bm u_0 |_{\Omega_{\mathrm{in}}} \text{ in } \bm H^2(\Omega_{\mathrm{in}}), \quad \bm u_{0\ep} |_{\Omega_{\mathrm{out}}} \rightharpoonup \bm u_0 |_{\Omega_{\mathrm{out}}} \text{ in } \bm H^2(\Omega_{\mathrm{out}}).
\]
%
% The $H^2$ regularity of $\bm u_0|_{\Omega_{\mathrm{in}}}$ and $\bm u_0|_{\Omega_{\mathrm{out}}}$ implies the $H^1$ regularity of $p_0|_{\Omega_{\mathrm{in}}}$ and $p_0|_{\Omega_{\mathrm{out}}}$ respectively. 
By using the regularity result on the Dirichlet problem of the Stokes equations, we conclude that
\begin{equation}\label{eq:u0-p0-H2}
\|\bm u_0\|_{\bm H^2(\Omega_a)} + \|p_0\|_{\bm H^1(\Omega_a)} \le C(\|g(0)\|_{H^1(\Gamma)}+\|\bm h\|).
\end{equation}
\end{proof}

\end{document}